\documentclass[a4paper, leqno]{article}
\usepackage[english]{babel}
\usepackage{amsmath,amscd,amssymb,amsthm}
\usepackage{physics}
\usepackage{pb-diagram}
\usepackage{cancel}

\newtheorem{definition}{Definition}[section]
\newtheorem{theorem}{Theorem}[section]
\newtheorem{lemma}{Lemma}[section]
\newtheorem{proposition}{Proposition}[section]
\newtheorem{corollary}[theorem]{Corollary}

\begin{document}

\title{Deformations of holomorphic-Higgs pairs}
\author{Takashi Ono}

\date{}
\maketitle

\begin{abstract}

Let $X$ be a complex manifold and $(E,\theta)$ be a Higgs bundle over $X.$ We study the deformation of holomorphic-Higgs pair $(X, E, \theta)$. We introduce the differential graded Lie algebra (DGLA) which comes from the deformation. We derive the Maurer-Cartan equation which governs the deformation of the holomorphic-Higgs pair, construct the Kuranishi family of it, and prove its local completeness.


\end{abstract}

\section{Introduction}
Let $X$ be a compact complex manifold and $(E,\overline\partial_E)$ be a holomorphic vector bundle on it.  Let $\overline\partial_{\mathrm{End}(E)}$ be the natural holomorphic structure on $\mathrm{End}(E)$ induced  by $E$. Let $A^{1,0}(\mathrm{End}(E))$ be the smooth sections of $\mathrm{End}(E)\otimes \Omega^{1,0}$. A Higgs field $\theta$ on $(E,\overline\partial_E)$ is an additional structure on $E$ such that, $\theta\in A^{1,0}(\mathrm{End}(E))$, $\overline\partial_{\mathrm{End}(E)} \theta=0$ and satisfies the integrability condition $\theta\wedge\theta=0.$  Higgs field was introduced in \cite{Hit} for the Riemann surfaces case and generalized to the higher dimensional case in \cite{Sim}.\par
This paper studies the deformation of pair $(X, E,\theta)$. The theory of deformation of pair for $(X, E)$ has been studied algebraically in \cite{Huy, Li, M1, Ser}, analytically in \cite{Hua, Siu} and in the style of Kodiara-Spencer \cite{Chan}. The deformation of pair $(X, E,\theta)$ was studied algebraically in \cite{M2}. Here we study this deformation in the style of Kodaira-Spencer \cite{Kod-Spe 1, Kod-Spe 2, Ku}. We use differential-geometric notions such as curvature, connection, and differential operators. We obtain similar results to the classical one.\par
Before we introduce our results, we recall some facts about the deformation theory for compact complex manifolds. Let $X$ be a compact complex manifold and $\Delta$ be a small ball centered at the origin of $\mathbb{C}^d$. Let $\{X_t\}_{t\in\Delta},X=X_0$ be a deformation family of $X$. By the standard argument of the deformation of a complex manifold, we have a family of Maurer-Cartan elements $\{\phi_t\}_{t\in\Delta}\subset A^{0,1}(TX)$ such that each $\phi_t$ determines the complex structure of $X_t$. Here $TX$ is the holomorphic tangent bundle of $X$. Let $\overline\partial$ be the Dolbeault operator on $X$ and $l_{\phi_t}=\partial(\phi_t\lrcorner)-\phi_t\lrcorner\partial $. While the Dolbeault operator $\overline\partial_t$ on $X_t$ is difficult to write down explicitly, we have a much more convenient operator,
\begin{equation*}
\overline\partial+l_{\phi_t}:A^{p,0}(X)\to A^{p,1}(X).
\end{equation*} 
 Although $\overline\partial+l_{\phi_t}$ is not same to $\overline\partial_t$, their kernel coincides. Hence we can check a function or a differential form is holomorphic or not for $X_t$  by only using the complex structure on $X$.\par
We have similar results in our case.  Once we consider the deformation of the holomorphic-Higgs pair over a small ball $\Delta$, we obtain a family of holomorphic-Higgs pair $\{(X_t,E_t,\theta_t)\}_{t\in\Delta}$ and $\{X_t\}_{t\in \Delta}$ is a deformation family of $X$. Hence we obtain a family of Maurer-Cartan element $\{\phi_t\}_{t\in\Delta}\subset A^{0,1}(TX)$ which each $\phi_t$ determines the $X_t$'s complex structure. Combining $\phi_t$, the holomorphic structure $\overline\partial_{E_t}$ on $E_t$ and the Higgs field $\theta_t$, we can define a differential operator $\overline D_t:A(X)\to A^{1}(X)$. Let $K$ be a hermitian metric on $E$, $\partial_K$ be the (1,0)-part of the Chern connection which is uniquely determined by $\overline\partial_E$ and $K$. $\{\partial_K,\phi_t\lrcorner\}$ is the operator such that $\{\partial_K,\phi_t\lrcorner\}=\partial_K(\phi_t\lrcorner)-\phi_t\lrcorner\partial_K$. $\overline D_t$ defines a family of elements  
\begin{equation*}
A_t:=\overline D_t-\overline\partial_E-\{\partial_K,\phi_t\lrcorner\}-\theta\in A^1(\mathrm{End}(E)).
\end{equation*}
 
Conversely, if we have a given family $A_t\in A^{0,1}(\mathrm{End}(E)),B_t\in A^{1,0}(\mathrm{End}(E))$ and $\phi_t\in A^{0,1}(TX)$, we can define a differential operator as \begin{equation*}
\overline D_t:=\overline\partial_E+\{\partial_K,\phi_t\lrcorner\}+A_t +\theta + B_t.
\end{equation*}
We introduce the first result of this paper. This result can be considered the Newlander-Nirenberg type theorem for deformations of pair $(X, E,\theta)$. 

\begin{proposition}[Proposition \ref{N-N}]
Assume that we have a smooth family $A_t\in A^{0,1}(\mathrm{End}(E)),B_t\in A^{1,0}(\mathrm{End}(E))$ and $\phi_t\in A^{0,1}(TX)$ parametrized by $t\in \Delta$. Let $A_0=B_0=0$ and $\phi_0=0.$ Let $\overline D_t$ be a differential operator as above. If $\overline D_t^2=0,$ then $(A_t,B_t,\phi_t)$  defines a  holomorphic-Higgs pair $(X_t,E_t,\theta_t)$(i.e. a complex manifold $X_t$ and Higgs bundle ($E_t,\theta_t)$ on it).
\end{proposition}

Applying this result we can construct the differential graded Lie algebra (DGLA) $(L,d_{T(E)},[\cdot,\cdot]_{T(E)})$ and derive the $\textit{Maurer-Cartan equation}$ which governs the deformation of holomorphic-Higgs pair (See Section \ref{DGLA section} for more details). Let $\partial_{K}^{\mathrm{End}(E)}:A(\mathrm{End}(E))\to A^{1,0}(\mathrm{End}(E))$ be the differential operator induced by $\partial_K$ and for $\phi\in A^{0,i}(TX),$ $\{\partial_{K}^{\mathrm{End}(E)},\phi\lrcorner\}:=\partial_{K}^{\mathrm{End}(E)}(\phi\lrcorner)+(-1)^i\phi\lrcorner\partial_{K}^{\mathrm{End}(E)}$. Let $[\cdot,\cdot]$ be the standard Lie bracket on $A^*(\mathrm{End}(E))=\oplus_iA^i(\mathrm{End}(E))$ and  $[\cdot,\cdot]_{SN}$ be the standard Schouten-Nijenhuys bracket on $A^*(TX)=\oplus_iA^{0,i}(TX)$.
\begin{theorem}[Theorem \ref{DGLA}]
Let $L^i=\oplus_{p+q=i}A^{p,q}(\mathrm{End}(E))\oplus A^{0,i}(TX)$ and $L:=\oplus_i L^i.$ Let $(A,\phi)\in L^i$ and $(B,\psi)\in L^j.$
We set,
\begin{equation*}
[(A,\phi),(B,\psi)]_{T(E)}:=((-1)^i\{\partial_{K}^{\mathrm{End}(E)},\phi\lrcorner\}B-(-1)^{(i+1)j}\{\partial_{K}^{\mathrm{End}(E)},\psi\lrcorner\}A-[A,B],[\phi,\psi]_{SN})
\end{equation*}
Let $B\in A^{0,1}(Hom(TX,\mathrm{End}(E))$ and $C\in A^{1,0}(Hom(TX,\mathrm{End}(E))$ and they acts on $v\in A^{0,p}(TX)$ as,
\begin{equation*}
B(v):=(-1)^{p}v\lrcorner F_{d_{K}}, C(v):=\{\partial_{K}^{\mathrm{End}(E)},v\lrcorner\}\theta.
\end{equation*}
We define $d_{T(E)}:L\to L$ as,
\begin{equation*}
d_{T(E)}:=
\begin{pmatrix}
\overline\partial_{\mathrm{End}(E)} & B\\
0 & \overline\partial_{TX} 
\end{pmatrix}
+
\begin{pmatrix}
\theta & C\\
0 & 0
\end{pmatrix}.
\end{equation*}
Then $(L,d_{T(E)},[\cdot,\cdot]_{T(E)})$ is a DGLA.
\end{theorem}
Combining Proposition 1.1 and Theorem 1.1, we have, 
\begin{theorem}[Theorem \ref{DGLA for Higgs bundle}]
Let $(A_t,B_t,\phi_t)$ be as above. Then $(A_t,B_t,\phi_t)$ defines a holomorphic-Higgs pair if and only if $(A_t,B_t,\phi_t)$ satisfies the Maurer-Cartan equation:
\begin{equation*}
d_{T(E)}(A_t+B_t)-\dfrac{1}{2}[(A_t+B_t,\phi_t),(A_t+B_t,\phi_t)]_{T(E)}=0.
\end{equation*}
\end{theorem}
In this paper, the differential $d_{T(E)}$ and the bracket $[\cdot,\cdot]_{T(E)}$ is constructed differential geometrically. The advantage of constructing $d_{T(E)}$ and  $[\cdot,\cdot]_{T(E)}$ differential geometrically is that we can write them down explicitly. Hence we can apply the technique of \cite{Kod-Spe 1,Kod-Spe 2,Ku} to construct the universal family (= $\mathit{Kuranishi}$ $\mathit{famliy}$) for a pair $(X,E,\theta)$.\par
Let $\Delta_{T(E)}$ be the laplacian induced by $d_{T(E)}$.  Since $\Delta_{T(E)}$ is an elliptic operator, $\mathbb{H}^i:=\mathrm{ker}(\Delta_{T(E)}:L^i\to L^i)$ is finite dimension. Let $\{\eta_1,\dots,\eta_n\}$ be a basis of $\mathbb{H}^1$. Let $d_{T(E)}^*$ be the formal adjoint of $d_{T(E)}$, $H:L^i\to\mathbb{H}^i$ be the projection and $G$ be the Green operator associated to $\Delta_{T(E)}.$ The next result is based on Kuranishi \cite{Ku}.
\begin{proposition}[Propostion 4.1, 4.2]
Let $t=(t_1,\dots,t_n)\in\mathbb{C}^n$ and $\epsilon_1(t):=\sum_i t_i\eta_i.$ For all $|t|<<1$ we have a $\epsilon(t)$ such that $\epsilon(t)$ satisfies the following equation:
\begin{equation*}
\epsilon(t)=\epsilon_1(t)+\dfrac{1}{2}d_{T(E)}^*G[\epsilon(t),\epsilon(t)]_{T(E)}.
\end{equation*}
Moreover, $\epsilon(t)$ is holomorphic respect to the variable $t$ and $\epsilon(t)$ satisfies the Maurer-Cartan equation
\begin{equation*}
d_{T(E)}\epsilon(t)-\dfrac{1}{2}[\epsilon(t),\epsilon(t)]_{T(E)}=0
\end{equation*}
if and only if $H[\epsilon(t),\epsilon(t)]_{T(E)}=0.$
\end{proposition}
 Let $\Delta\subset\mathbb{C}^n$ be a small ball such that $\epsilon(t)$ is holomorphic on $\Delta.$  We define $\mathcal{S}\subset\Delta$ as 
 \begin{equation*}
 \mathcal{S}:=\{t\in\Delta~ |~ H[\epsilon(t),\epsilon(t)]_{T(E)}=0\}.
 \end{equation*}
$\mathcal{S}$ might not be smooth, however, it is a complex analytic space. Let $X_{\epsilon(t)}, E_{\epsilon(t)}, \theta_{\epsilon(t)}$ be the complex manifold, the holomorphic bundle, and the Higgs field which ${\epsilon(t)}$ defines. By combining the above results we have a family of holomorphic-Higgs pair $\{(X_{\epsilon(t)}, E_{\epsilon(t)},\theta_{\epsilon(t)})\}_{t\in\mathcal{S}}$. We call this family the \textit {Kuranishi family} of $(X, E,\theta)$ and $\mathcal{S}$ the \textit{Kuranishi space}. \par
For a complex manifold $X$, the Kuranishi family has the universal property such that the Kuranishi family contains all small deformation of $X$ (See \cite{Cap, Doua a,Doua b,Ku} for more details). However, in this paper, we only prove the abbreviated version of the completeness of Kuranishi space. The next theorem states that the constructed Kuranishi family has a property of local completeness.\par
Let $|\cdot|_k$ be the $k$-th Sobolev norm on $L^1$. We assume $k>>1$.

\begin{theorem}[Theorem \ref{main}]
Let $(X, E,\theta)$ be a holomorphic-Higgs pair. Let $\mathcal{S}$ be a Kuranishi family for $(X,E,\theta)$.  Let $\eta\in L^1$ be a Maurer-Cartan element. If $|\eta|_k$ is small enough, then there is a $t\in \mathcal{S}$ such that $(X_\eta, E_\eta,\theta_\eta)$ is isomorphic to $(X_{\epsilon(t)}, E_{\epsilon(t)},\theta_{\epsilon(t)})$ (See section 4 for the meaning of isomorphic).
\end{theorem}

\subsection*{Acknowledgement} 
The author would like to express his gratitude to his supervisor Hisashi Kasuya for his encouragement and patience. The author thanks Masataka Iwai and Takahiro Saito for their kindness and encouragement.   

\section{Deformation of holomorphic-Higgs pair}

\begin{definition}
Let $X$ be a compact complex manifold. Let $\overline{\partial}_{\mathrm{End}(E)}$ be the complex structure on $\mathrm{End}(E)$ induced by $E$. A Higgs bundle $(E,\theta)$ over $X$ is a pair such that, 
\begin{itemize}
      \item$E$ is a holomorphic bundle over $X$.
      \item$\theta$ is a Higgs filed that $\theta\in A^{1,0}(\mathrm{End}(E)),\overline{\partial}_{\mathrm{End}(E)}\theta=0,\theta\land\theta=0$.
\end{itemize}
We call a pair $(X, E,\theta)$ a holomorphic-Higgs pair.
\end{definition}
We fix a metric $K$ on $E$ and assume $X$ to be compact throughout this paper.
\begin{definition}
Let $(X, E,\theta)$ be a holomorphic-Higgs pair. A family of deformation of holomorphic-Higgs pair over a small ball $\Delta$ centered at the origin of $\mathbb{C}^{d}$, consists of a complex manifold $\mathcal{X}$, a proper submersive holomorphic map
$$\pi : \mathcal{X}\to\Delta$$
and a Higgs Bundle $(\mathcal{E},\Theta)$ over $\mathcal{X}$ such that, $\pi^{-1}(0)=X$, $\mathcal{E}|_{\pi^{-1}(0)}=E$, $\Theta|_{\pi^{-1}(0)}=\theta$.
\end{definition}

By Ehresmann's theorem and as in \cite[Chapter 7 Lemma 7.1]{Kod}, if we choose $\Delta$ small enough, we have maps $F: X\times\Delta\to\mathcal{X}$ and $P: E\times\Delta\to\mathcal{E}$ such that the all diagrams below commutes, $F$ is a diffeomorphism and $P$ is a smooth bundle isomorphism. 
\[
\begin{diagram}
    \node{E\times\Delta} \arrow{e,t}{P} \arrow{s} \node{\mathcal{E}} \arrow{s} \\
    \node{X\times\Delta} \arrow{e,t}{F}\arrow{s} \node{\mathcal{X}}\arrow{sw}\\
    \node[1]{\Delta}
\end{diagram}
\]
We can induce a complex structure on $X\times\{t\}$ and a Higgs bundle structure on $E\times\{t\}$ using $F|_{X\times\{t\}}:X\times\{t\}\to\mathcal{X}_{t}=\pi^{-1}(t)$ and $P|_{E\times\{t\}}:E\times\{t\}\to\mathcal{E}|_{\pi^{-1}(t)}$. We denote this family of holomorphic-Higgs pair $\{(X_{t},E_{t},\theta_{t})\}_{t\in\Delta}$.\par
  Since $\{X_{t}\}_{t\in\Delta}$ is a deformation of the complex manifold $X$, we have a family of Maurer-Cartan element $\{\phi_{t}\}_{t\in\Delta}$ such that each $\phi_{t}$ determines the complex structure of $X_{t}$. \par
  Let $A^{1,0}(X_{t}):=\{\alpha\in A^{1} ~|~ \alpha$ is a (1,0)-form of $X_{t}$$\}$ and $\pi_{X}^{(1.0)}:A^1(X)\to A^{(1,0)}(X)$, $\pi_{X}^{(0,1)}:A^{1}(X)\to A^{(0,1)}(X)$ be the natural projection. 
 
 \begin{lemma}\label{1,0-form}
 $\alpha\in A^{1,0}(X_{t})$ if and only if   $\pi_{X}^{(0.1)}(\alpha)=\phi_{t} \lrcorner\pi_{X}^{(1.0)}(\alpha)$
 \end{lemma}
\begin{proof}
It is enough to prove it locally. Let $x\in X$ and $U_x$ be an open neighborhood of $x$. Let $(\xi_1,\dots,\xi_n),(z_1,\dots,z_n)$ be local coordinates on $U_x$ and $(\xi_1,\dots,\xi_n)$ be a complex coordinate for $X_t$ and $(z_1,\dots,z_n)$ be a complex coordinate for $X$.\par
Let $\alpha = \sum_i f_i d\xi_i$. We have $\pi_X^{0,1}(\alpha)=\sum_{i,j}f_i \dfrac{\partial\xi_i}{\partial \overline z_j}d\overline z_j$ and $\pi_X^{1,0}(\alpha)=\sum_{i,j}f_i \dfrac{\partial\xi_i}{\partial z_j}dz_j$.\par
Recall that $\phi_t=\sum_{i,j} \phi^i_{t,j}\dfrac{\partial}{\partial z_i}\otimes d\overline z_j,(\phi^i_{t,j})=\biggl(\dfrac{\partial\xi_i}{\partial z_k}\biggr)^{-1}\biggl(\dfrac{\partial \xi_k}{\partial\overline z_j}\biggr).$ See for more details.\par
Hence,
\begin{equation*}
\phi_t\lrcorner\pi^{1,0}_X(\alpha)=\biggl(\sum_{j,k}\phi^j_{t,k}\dfrac{\partial}{\partial z_j}\otimes d\overline z_k\biggr)\lrcorner\biggl(\sum_{i,j}f_i \dfrac{\partial\xi_i}{\partial z_j}dz_j\biggr)=\sum_{i,j,k}f_i \dfrac{\partial\xi_i}{\partial z_j}\phi^j_{t,k}d\overline z_k=\sum_{i,k}f_i\dfrac{\partial\xi_i}{\partial\overline z_k}d\overline z_k=\pi_X^{0,1}(\alpha).
\end{equation*}

To prove the converse, we only have to prove that if $\omega\in A_{X_t}^{0,1}$ and $\pi_{X}^{(0.1)}(\omega)=\phi_{t} \lrcorner\pi_{X}^{(1.0)}(\omega)$ stands, $\omega=0$.
Let $\omega=\sum_i h_i d\overline \xi_i$ and assume $\pi_{X}^{(0.1)}(\omega)=\phi_{t} \lrcorner\pi_{X}^{(1.0)}(\omega)$. We have $\pi_X^{0,1}(\omega)=\sum_{i,j}h_i\dfrac{\partial \overline \xi_i}{\partial \overline z_j}d\overline z_j$ and $\phi_t\lrcorner\pi^{1,0}_X(\omega)=\sum_{i,j,k}h_i \dfrac{\partial\overline\xi_i}{\partial z_j}\phi^j_{t,k}d\overline z_k$. Since $\pi_{X}^{(0.1)}(\omega)=\phi_{t} \lrcorner\pi_{X}^{(1.0)}(\omega)$, we have,
\begin{equation*}
0=\pi_{X}^{(0.1)}(\omega)-\phi_{t} \lrcorner\pi_{X}^{(1.0)}(\omega)=\sum_k\biggl \{\sum_{i}h_i\biggl(\dfrac{\partial\overline\xi_i}{\partial\overline z_k}-\sum_{j}\dfrac{\partial\overline\xi_i}{\partial z_j}\phi^j_{t,k}\biggr)\biggr \} d\overline z_k.
\end{equation*}
Hence,
\begin{equation*}
\biggl(\dfrac{\partial\overline\xi_i}{\partial\overline z_k}-\sum_{j}\dfrac{\partial\overline\xi_i}{\partial z_j}\phi^j_{t,k}\biggr) (h_1,\dots,h_n)^T=0.
\end{equation*}
Since $\phi_t$ defines a near complex structure with respect to the original one,
\begin{equation*}
\mathrm{det}\biggl(\dfrac{\partial\overline\xi_i}{\partial\overline z_k}-\sum_{j}\dfrac{\partial\overline\xi_i}{\partial z_j}\phi^j_{t,k}\biggr)\neq 0.
\end{equation*}
Hence $(h_1,\dots,h_n)=0$. This implies  $\omega=0$.
\end{proof}

 \begin{lemma}\label{hol 1 form}
 Let $\alpha\in A^{1,0}(X_{t})$. $\alpha$ is a holomorphic 1-form on $X_{t}$ if and only if $(\overline{\partial}+l_{\psi_{t}})\pi_{X}^{(1.0)}(\alpha)=0$.
 Here $l_{\psi_{t}}=\partial(\psi_{t} \lrcorner)-\psi_{t} \lrcorner\partial$.
\end{lemma}
\begin{proof}
As in lemma \ref{1,0-form}, we only have to prove it locally. We use the notation we used in the proof of lemma  \ref{1,0-form}.\par
Let  $\alpha = \Sigma_i f_i d\xi_i$ and let $\alpha^{1,0}=\pi_{X}^{(1.0)}(\alpha)$. We first calculate $(\overline{\partial}+l_{\psi_{t}})(\alpha^{1,0}).$ Since $\alpha^{1,0}=\sum_{i,j}f_i\dfrac{\partial\xi_i}{\partial z_j}dz_j$, we have
\begin{equation*}
\overline\partial\alpha^{1,0}=\overline\partial\biggl(\sum_{i,j}f_i\dfrac{\partial\xi_i}{\partial z_j}dz_j\biggr)=\sum_{i,j,k}\bigg\{\dfrac{\partial f_i}{\partial\overline z_k}\dfrac{\partial\xi_i}{\partial z_j}+f_i\dfrac{\partial^2\xi_i}{\partial\overline z_k\partial z_j}\biggr\}d\overline z_k\wedge dz_j
\end{equation*}
\begin{align*}
l_{\phi_t}(\alpha^{1,0})&=l_{\phi_t}\bigg(\sum_{i,j}f_i\dfrac{\partial\xi_i}{\partial z_j}dz_j\bigg)=\partial\biggl(\phi_t\lrcorner\sum_{i,j}f_i\dfrac{\partial\xi_i}{\partial z_j}dz_j\biggr)-\phi_t\lrcorner\biggl\{\sum_{i,j,k}\bigg(\dfrac{\partial f_i}{\partial z_k}\dfrac{\partial\xi_i}{\partial z_j}+f_i\dfrac{\partial^2\xi_i}{\partial z_k\partial z_j}\biggr)dz_k\wedge dz_j\biggr\}\\
&=\partial\biggl(\sum_{i,j,k}f_i\dfrac{\partial\xi_i}{\partial z_j}\phi^j_{t,k}d\overline z_k\biggr)-\phi_t\lrcorner\biggl\{\sum_{i,j,k}\bigg(\dfrac{\partial f_i}{\partial z_k}\dfrac{\partial\xi_i}{\partial z_j}\bigg)dz_k\wedge dz_j\bigg\}\bigg.
\end{align*}
\textrm{and}
\begin{align*}
\partial\biggl(\sum_{i,j,k}f_i\dfrac{\partial\xi_i}{\partial z_j}\phi^j_{t,k}d\overline z_k\biggr)=\partial\bigg(\sum_{i,k}f_i\dfrac{\partial\xi_i}{\partial\overline z_k}d\overline z_k\bigg)=\sum_{i,j,k}\dfrac{\partial f_i}{\partial z_j}\dfrac{\partial \xi_i}{\partial\overline z_k}dz_j\wedge d\overline z_k+\sum_{i,j,k}f_i\dfrac{\partial^2\xi_i}{\partial\overline z_k\partial z_j}dz_j\wedge d\overline z_k
\end{align*}
\begin{align*}
&=\sum_{i,j,k}\dfrac{\partial f_i}{\partial z_k}\dfrac{\partial\xi_i}{\partial z_j}\sum_l\phi^k_{t,l}d\overline z_l\wedge dz_j-\sum_{i,j,k}\dfrac{\partial f_i}{\partial z_k}\dfrac{\partial\xi_i}{\partial z_j}\sum_l\phi^j_{t,l}d\overline z_l\wedge dz_k\\
&=\sum_{i,j,k,l}\dfrac{\partial f_i}{\partial z_k}\dfrac{\partial\xi_i}{\partial z_j}\phi^k_{t,l}d\overline z_l\wedge dz_j-\sum_{i,k,l}\dfrac{\partial f_i}{\partial z_k}\dfrac{\partial\xi_i}{\partial\overline z_l}d\overline z_l\wedge dz_k.
\end{align*}
Hence,
\begin{align}\label{dbar alpha}
(\overline{\partial}+l_{\phi_{t}})(\alpha^{1,0})&=\notag\sum_{i,j,k}\dfrac{\partial f_i}{\partial\overline z_k}\dfrac{\partial\xi_i}{\partial z_j}d\overline z_k\wedge dz_j+\sum_{i,j,k}f_i\dfrac{\partial^2\xi_i}{\partial\overline z_k\partial z_j}d\overline z_k\wedge dz_j\\
&\notag+\sum_{i,j,k}\dfrac{\partial f_i}{\partial z_j}\dfrac{\partial \xi_i}{\partial\overline z_k}dz_j\wedge d\overline z_k+\sum_{i,j,k}f_i\dfrac{\partial^2\xi_i}{\partial\overline z_k\partial z_j}dz_j\wedge d\overline z_k\\
&\notag-\sum_{i,j,k,l}\dfrac{\partial f_i}{\partial z_k}\dfrac{\partial\xi_i}{\partial z_j}\phi^k_{t,l}d\overline z_l\wedge dz_j+\sum_{i,k,l}\dfrac{\partial f_i}{\partial z_k}\dfrac{\partial\xi_i}{\partial\overline z_l}d\overline z_l\wedge dz_k\\
&\notag=\sum_{i,j,l}\dfrac{\partial f_i}{\partial\overline z_l}\dfrac{\partial\xi_i}{\partial z_j}d\overline z_l\wedge dz_j-\sum_{i,j,k,l}\dfrac{\partial f_i}{\partial z_k}\dfrac{\partial\xi_i}{\partial z_j}\phi^k_{t,l}d\overline z_l\wedge dz_j\\
&=\sum_{l,l}\sum_{i}\dfrac{\partial\xi_i}{\partial z_j}\bigg(\dfrac{\partial f_i}{\partial\overline z_l}-\sum_{k}\dfrac{\partial f_i}{\partial z_k}\phi^k_{t.l}\bigg)d\overline z_l\wedge dz_j.
\end{align}
If we assume $\alpha$ to be a holomorphic 1 form on $X_t$, this implies that $\{f_i\}_i$ are holomorphic functions on $X_t$. Hence we have,
\begin{equation*}
 \dfrac{\partial f_i}{\partial\overline z_l}-\sum\dfrac{\partial f_i}{\partial z_k}{\phi^k_{t,l}}=0.
\end{equation*}
Hence by $(\ref{dbar alpha})$, when $\alpha$ is a holomorphic 1 form on $X_t$, ($\overline{\partial}+l_{\phi_{t}})(\alpha^{1,0})=0$.\par
Conversely, if we assume ($\overline{\partial}+l_{\phi_{t}})(\alpha^{1,0})=0$, by $(\ref{dbar alpha})$ we have 
\begin{equation*}
\bigg(\dfrac{\partial\xi_i}{\partial z_j}\bigg)\bigg(\dfrac{\partial f_i}{\partial\overline z_l}-\sum_{k}\dfrac{\partial f_i}{\partial z_k}\phi^k_{t.l}\bigg)=0.
\end{equation*}
Since $\phi_t$ defines a near complex structure to $X$, we have $\mathrm{det}\bigg(\dfrac{\partial\xi_i}{\partial z_j}\bigg)\neq0$.
 Hence $\bigg(\dfrac{\partial f_i}{\partial\overline z_l}-\sum_{k}\dfrac{\partial f_i}{\partial z_k}\phi^k_{t,l}\bigg)=0.$ This shows that $\{f_i\}$  are holomorphic function on $X_t$ and $\alpha$ is a holomorphic 1 form on $X_t$.
\end{proof}
 By Lemma \ref{1,0-form}, $\theta_{t}$ can be decomposed as $\theta_{t}=\omega_{t}+\phi_{t}\lrcorner \omega_{t}$, where $\omega_{t}=\pi_{X}^{(1.0)}(\theta_{t})$. 
 We define an operator $\overline D_{t}: A(E)\to A^1(E)$ as follows.
 $$s\in A(E), \overline D_{t}(s)=\overline D_{t}(s^ke_k):=(\partial + l_{\phi_{t}})s^k\otimes e_k + \omega_{t}\wedge s$$
Here, $\{e_k\}$ is a local holomorphic frame of $E_{t}$ and we used the Einstein summation rule. 
\begin{proposition}
$\overline D_{t}$ is a well defined operator, that is, $\overline D_{t}$  is independent of the holomorphic frame of $E_{t}$.
Also $\overline D_{t}$ satisfies the Leibniz rule. :
      \begin{center}
       $\overline D_{t}(\alpha \wedge s)=(\overline{\partial}+l_{\phi_{t}})\alpha\otimes s + (-1)^{p}\alpha\wedge\overline D_{t}(s)$
      \end{center}
      for every $\alpha\in A^p(X)$ and $s\in A(E)$.
\end{proposition}
\begin{proof}
To prove well-definedness, we need to show that $\overline D_{t}$ is independent of the choice of a local holomorphic frame $\{e_{k}\}$ of $E_{t}.$ Take an another local holomorphic frame $\{f_{j}\}$ of $E_{t}$. Let $h^{j}_{k}$ be a holomorphic functions of $X_{t}$ such that $f_{j}=h^{k}_{j}e_{j}$. Then for local section $s\in A(E), s=\widetilde{s}^{j} f_{j}=s^k e_{k}$, we have $\widetilde{s}_{j}h^{k}_{j}=s_{k}$, thus we have 
\begin{equation*}
\begin{split}
\overline D_{t}(\widetilde{s}^{j}f_{j})&=(\overline\partial + l_{\phi_{t}})\widetilde{s}^{j}\otimes f_{j} + \omega_{t}\wedge(\widetilde{s}^{j}f_{j})\\
&=(\overline\partial + l_{\phi_{t}})\widetilde{s}^{j}\otimes h^k_j e_{k} +\omega_{t}\wedge s \\
&=(\overline\partial + l_{\phi_{t}})(\widetilde{s}^{j}h^{k}_{j})\otimes e_{k} + \omega_{t}\wedge (s^k e_k)\\
&=(\overline\partial + l_{\phi_{t}})(s^k)\otimes e_k +\omega_{t}(s^k e_k)\\
&=\overline D_{t}(s^k e_{k})
\end{split}
\end{equation*}
Hence $\overline D_{t}$ is well defined.\par
The Leibneiz rule for $\overline D_t$ follows from the fact that  $\alpha\in A^p(X),\beta\in A^q(X),\alpha\wedge\beta=(-1)^{pq}\beta\wedge\alpha$ stands and $\overline\partial + l_{\phi_{t}}$ satisfies the Leibniz rule.:
\begin{center}
$(\overline\partial + l_{\phi_{t}})(\alpha\wedge\beta)=(\overline\partial + l_{\phi_{t}})(\alpha)\wedge\beta + (-1)^p\alpha\wedge(\overline\partial + l_{\phi_{t}})(\beta).$
\end{center}
\end{proof}

\begin{proposition}
${\overline D_{t}}^{2}=0.$
\end{proposition}
\begin{proof}
We calculate ${\overline D_{t}}^{2}$ locally and show ${\overline D_{t}}^{2}=0.$ Since $\overline D_t$ satisfies the Leibniz rule, we only have to prove $(\overline\partial+l_{\phi_t})^2= 0$ and ${\overline D_{t}}^{2}(s)=0$ for $s\in A(E).$
\par
 First we prove $(\overline\partial+l_{\phi_t})^2=0.$ According to \cite{M2}, we have,
 \begin{equation*}\label{square}
 (\overline\partial+l_{\phi_t})^2=l_{\overline\partial_{TX}\phi_t-\frac{1}{2}[\phi_t,\phi_t]}.
 \end{equation*}
 Since $\phi_t$ is a Maurer-Cartan element, we have $\overline\partial_{TX}\phi_t-\dfrac{1}{2}[\phi_t,\phi_t]=0.$ Hence $(\overline\partial+l_{\phi_t})^2$=0.
 \par
 Next we prove ${\overline D_{t}}^{2}(s)=0$ for $s\in A(E).$ Let  $\{e_k\}$ be a holomorphic frame for $E_t.$ Assume that   $s$ and $\omega_t$ has a trivialization as $s=s^ke_k$ and  $\omega_t=g_idz_i,g_i=(a^s_{i,t})$ respect to the frame $\{e_k\}.$ Here  $s^k,a^s_{i,t}\in A(X)$ and $g_i\in A(EndE).$ Since $\omega_t=\pi_X^{1,0}(\theta_t)$  and $\theta_t$ is a Higgs field we have $\omega_t\wedge\omega_t=0.$ Applying lemma \ref{hol 1 form} and the fact that $\overline D_t$ satisfies the Leibniz rule, we have,
\begin{equation*}
\begin{split}
\overline D_t^2( s)=\overline D_t^2( s^k\otimes e_k )&=\overline D_t((\overline\partial+l_{\phi_t})s^k\otimes e_k+\omega_t\wedge s)\\
&=(\overline\partial+l_{\phi_t})^2s^k\otimes e_k+\omega_t\wedge(\overline\partial+l_{\phi_t})s^k\otimes e_k+(\overline\partial+l_{\phi_t})(a^s_{i,k}s^kdz_i)\otimes e_s+\omega_t\wedge\omega_t\wedge s\\
&=\omega_t\wedge(\overline\partial+l_{\phi_t})s^k\otimes e_k+(\overline\partial+l_{\phi_t})(a^s_{i,k}dz_i)\wedge s^k\otimes e_s-\omega_t\wedge(\overline\partial+l_{\phi_t})s^k\otimes e_k+\omega_t\wedge\omega_t\wedge s\\
&=0.
\end{split}
\end{equation*}
Since $s\in A(E)$ is an arbitrary smooth section, this proves the claim.
\end{proof}

\begin{proposition}
We define $A_{t}$ as $A_{t}:=\overline D_{t}-\overline \partial_{E}-\{\partial_{K},\phi_{t}\}-\theta$ then $A_{t}\in A^{1}(\mathrm{End}(E)).$ Here $\partial_K$ is a (1,0)-part of the Chern connection which is uniquely determined by $\overline\partial_E$ and the Hermitian metric $K$. $\{\partial_K,\phi_t\lrcorner\}$ is the operator such that $\{\partial_K,\phi_t\lrcorner\}=\partial_K(\phi_t\lrcorner)-\phi_t\lrcorner\partial_K$.
\end{proposition}
\begin{proof}
Let $f\in A(X)$ and $s\in A(E)$. Using the Leibniz rule and the fact that the contraction is only taken in the (1,0)-part, we have
\begin{equation*}
\begin{split}
A_{t}(fs)&=(\overline\partial + l_{\phi_{t}})f\otimes s + f\overline D_{t}(s) - \overline\partial f\otimes s-f\overline\partial_{E}s+\phi_{t}\lrcorner\partial_{K}(fs)-\theta\wedge(fs)\\
&=(\overline\partial  -\phi_{t}\lrcorner\partial)f\otimes s + f\overline D_{t}(s) - \overline\partial f\otimes s-f\overline\partial_{E}s+\phi_{t}\lrcorner(\partial f\otimes s +f\partial_{K}s)-f\theta\wedge s\\
&=f(\overline D_{t} - \overline\partial_{E} - \{\partial_{K},\phi_{t}\}-\theta)s\\
&=fA_{t}(s). 
\end{split}
\end{equation*}
This shows that $A_{t}\in A^{1}(EndE).$
\end{proof}
The above propositions tell us that if we have a family of holomorphic-Higgs pair $\{(X_{t}, E_{t},\theta_{t})\}_{t\in\Delta}$ which comes from the deformation of $(X, E,\theta)$, we naturally obtain a differential operator $\overline D_{t}$ such that $(\overline D_{t})^2=0$, and elements  $A_{t}\in A^{1}(\mathrm{End}(E))$ and  $\phi_{t}\in A_{X}^{(0,1)}(TX)$.
We want the converse.\par
Suppose we have a given smooth family $A_{t}\in A^{0,1}_{X}(\mathrm{End}E)$, $B_{t}\in A^{1,0}_{X}(\mathrm{End}E)$ and $\phi_{t}\in A^{0,1}_{X}(TX)$ parametrized by $t\in\Delta$.\par
We define the operator $\overline D_{t}:A(E)\to A^{1}(E)$ as
$$\overline D_{t}:=\overline\partial_{E} + \{\partial_{K},\phi_{t}\lrcorner\} + A_{t} + \theta + B_{t}.$$
We extend $\overline D_{t}$ to $A^{p}(E)$ in an obvious way so that it satisfies the Leibniz rule:
$$\overline D_{t}(\alpha\otimes s)=(\overline\partial + l_{\phi_{t}})\alpha\otimes s +(-1)^{p}\alpha\wedge\overline D_{t}(s).$$
 We want to show that if $\overline D_{t}^{2}= 0$, $(A_{t},B_{t},\phi_{t})$ defines a holomorphic-Higgs pair $(X_{t}, E_{t}, \theta_{t})$. First of all, we have,

\begin{proposition}
If $\overline D_{t}^{2}= 0$, $\phi_{t}$ defines a holomorphic structure on $X$. We denote this complex manifold $X_{t}$.
\end{proposition}
\begin{proof}
Let $f\in A(X)$ and $s\in A(E)$. Since  $\overline D_{t}^{2}= 0$ and it satisfies the Leibniz rule,
\begin{equation*}
0=\overline D_{t}^2(f\otimes s)=(\overline\partial + l_{\phi_{t}})^2f\otimes s.
\end{equation*}
Since $f$ and $s$ are arbitrary function and section, we have $(\overline\partial + l_{\phi_{t}})^2=0$. By \eqref{square}, we have, 
\begin{equation*}
0=(\overline\partial+l_{\phi_t})^2=l_{\overline\partial_{TX}\phi_t-\frac{1}{2}[\phi_t,\phi_t]}.
\end{equation*}
Hence $\overline\partial_{TX}\phi_t-\dfrac{1}{2}[\phi_t,\phi_t]=0.$ Hence $\phi_t$ defines a integrable complex structure on $X$.
\end{proof}

Next, we show that $E$ admits a holomorphic structure over $X_{t}$ and we can induce a Higgs field on it. Let us define $\overline D_{t}':A(E)\to A^{0,1}(E)$ as $\overline D_{t}':=\overline \partial_{E} + \{\partial_{K},\phi_{t}\lrcorner\} + A_{t}$. Remark that $\overline D_{t}= \overline D_{t}' + \theta + B_{t}$. The next claim was proved by the method in \cite{Mo}.

\begin{lemma}
$\mathrm{ker}(\overline D_{t}'$) generates $A(E)$ locally.
\end{lemma}
\begin{proof}
See the proof of \cite[Lemma 3.11.]{Chan}.
\end{proof}

The above lemma tells us that for every $x\in X$ we have an open neighborhood $U$ of $x$ and a frame $\{e_k\}$ on $U$ such that $\{e_k\}\subset \mathrm{ker}(\overline D_t')$.

Let $\{e_{k}\}$ be a local frame of $E$ such that $\{e_{k}\}\subset \mathrm{ker}(\overline D_{t}')$. Let $\overline\partial_t$ be the Dolbeault operator of $X_t.$ We can then define $\overline \partial_{E_{t}}$ by
$$\overline \partial_{E_{t}}(s^{k}e_{k}):=\overline\partial_{t}s^{k}\otimes e_{k}.$$
Let $\{f_{j}\}\subset \mathrm{ker}(\overline D_{t}')$ be an another local frame of $E$, then there exist $(h^{k}_{j})$ such that $f_{j}=h^{k}_{j}e_{k}$. Applying
$\overline D_{t}'$, we have,
$$\overline D_{t}'(f_{j})=\overline D_{t}'(h^{k}_{j}e_{j})=(\overline\partial - \phi_{t}\lrcorner \partial)h^{k}_{j}\otimes e_{k}$$
Since  $e_{k}$ is a local frame, we have $(\overline\partial - \phi_{t}\lrcorner \partial)h^{k}_{j}=0$, which is equivalent to $\overline\partial_{t}h^{k}_{j}=0$.

We can now check $\overline\partial_{E_{t}}$ is well defined. Let $s\in A(E)$ and assume $s$ has local trivilraization as $s=\widetilde{s}_{j}f_{j}=s_{k}e_{k}$.
Applying $\overline\partial_{E_{t}}$ we have,

\begin{center}
$\overline\partial_{E_{t}}(s^{k}e_{k})=\overline\partial_{t}s^{k}\otimes e_{k}=\overline\partial_{t}(\widetilde{s}_{j}h^{k}_{j})\otimes e_{k}=\overline\partial_{t}\widetilde{s}_{j}\otimes h^{k}_{j}e_{k}=\overline\partial_{t}\widetilde{s}_{j}\otimes f_{j}=\overline\partial_{E_{t}}(\widetilde{s}_{j} f_{j}).$
\end{center}

This proves the well-definedness. By definition, $\overline \partial_{E}$ satisfies the Leibniz rule:
\begin{center}
$\overline\partial_{E_{t}}(\alpha\otimes s)=\overline\partial_{t}\alpha\otimes s + (-1)^p\alpha\wedge \overline\partial_{E_{t}}s$
\end{center}

and $\overline\partial_{E_{t}}^2=0$ since $\phi_{t}$ defines an integral almost complex structure on $X$. Hence by the linearized version of the Newlander-Nirenberg Theorem, $E_{t}=(E,\overline\partial_{E_{t}})$ is a holomorphic bundle over $X_{t}$.\par
We want to show next that $\theta_{t}=\theta + B_{t} +\phi_{t} \lrcorner(\theta+B_{t})$ is a Higgs field for $E_{t}$ under the above assertion. By lemma \ref{1,0-form}, $\theta_{t}$ is a (1,0)-form of $X_{t}$ which takes value in $\mathrm{End}(E)$. \par

Let ${e_{k}}\subset \mathrm{ker}(\overline D_{t}')$ be a local frame of $E$ and assume $\theta + B_{t}$ is written as  $\theta + B_{t}=\Sigma_{i}g_{i}dz_{i} $ $(g_{i}\in A(\mathrm{End}(E)))$ respect to this frame. By lemma \ref{hol 1 form}, to show $\theta_{t}$ is a Higgs field on $E_{t}$, it is enough to show $(\overline\partial + l_{\phi_{t}})g_{i}dz_{i}=0$ and $(\theta + B_t)\wedge(\theta + B_t)$=0.\par
 Since $\overline D_{t}$ satisfies the Leibniz rule,
\begin{equation*}
\begin{split}
0&=\overline D_{t}^{2}(e_{k})=\overline D_{t}(\overline D_{t}(e_{k}))=\overline D_{t}((\theta + B_{t})(e_{k}))=\overline D_{t}(g_{i}dz_{i}(e_{k}))\\
&=(\overline\partial + l_{\phi_{t}})(g_{i}dz_{i})e_{k} - g_{i}dz_{i}\wedge \overline D_{t}(e_{k})\\
&=(\overline\partial + l_{\phi_{t}})(g_{i}dz_{i})e_{k}-(\theta+B_t)\wedge(\theta +B_t)(e_k)
\end{split}
\end{equation*}
Hence $\theta_{t}$ is a Higgs field for $E_{t}$ and $(X_{t},E_{t},\theta_{t})$ is a holomorphic-Higgs pair.

In summary, we have proved the following,

\begin{proposition}\label{N-N}
Suppose we have a given smooth family $A_{t}\in A^{0,1}_{X}(\mathrm{End}(E))$, $B_{t}\in A^{1,0}_{X}(\mathrm{End}(E))$, $\phi_{t}\in A_{X}^{0,1}(TX)$ parametrized by $t$. If the induced differential operator
$\overline D_{t}:A^{p}(E)\to A^{p+1}(E)$ satisfies $\overline D_{t}^{2}=0$ and the Leibniz rule 
$$\overline D_{t}(\alpha\wedge s)=(\overline{\partial}+l_{\phi_{t}})\alpha\otimes s + (-1)^{p}\alpha\wedge\overline D_{t}(s),$$

then $E$ admits a holomorphic structure over the complex manifold $X_{t}$, which we denote $E_{t}$ and a Higgs field $\theta_{t}$ such that $(X_{t},E_{t},\theta_{t})$ is a holmorphic-Higgs pair.
\end{proposition}

\section{DGLA and the Maurer-Cartan equation}\label{DGLA section}
Let us recall the definition of DGLA.\par
\begin{definition} A differential graded Lie algebra (DGLA) $(L,[\cdot,\cdot],d)$ is the date of $\mathbb{Z}$-graded vector space $L=\oplus_{i\in\mathbb{Z}}L^{i}$ with a bilinear bracket
 $[ \cdot,\cdot ]: L\times L\to L$ and a linear map $d$ such that
  \begin{enumerate}
  \item $a\in L^i, b\in L^j, ~[a,b]+(-1)^{ij}[b,a]=0.$
  \item The graded Jacobi identity holds:\par
  $a\in L^i,b\in L^j, c\in L^k,$ $[a,[b,c]]=[[a,b],c]+(-1)^{ij}[b,[a,c]].$
  \item $d(L^i)\subset L^{i+1},d\circ d=0$ and $ a\in L^i, b\in L^j,d[a,b]=[da,b]+(-1)^i[a,db].$ The map $d$ is called the differential of $L$.
\end{enumerate}
\end{definition}

We also recall the definition of the Maurer-Cartan equation of a DGLA. 

\begin{definition}
The Maurer-Cartan equation of a DGLA $L$ is 
\begin{center}
$da-\dfrac{1}{2}[a,a]=0,$ $a\in L^1$
\end{center}
The solutions of the Maurer-Cartan equation are called the Maurer-Cartan elements of the DGLA $L$.

\end{definition}
 We derive the Maurer-Cartan equation and DGLA which governs the deformation of the holomorphic-Higgs pair. The next proposition is important to construct the DGLA. Before we state it, we introduce some notation. Let $\partial_{K}^{\mathrm{End}(E)}:A(\mathrm{End}(E))\to A^{1,0}(\mathrm{End}(E))$ be the differential operator induced by $\partial_K$. Let $F_{d_K}$ be the curvature of the Chern connection. Let the bracket $[\cdot,\cdot]$ be the canonical Lie bracket defined on $A^*(\mathrm{End}(E))$ and $[\cdot,\cdot]_{SH}$ be the standard Schouten-Nijenhuys bracket defined on $A^{0,*}(TX)$.
 \begin{proposition}\label{vanish a}
 Suppose we have a $A\in A^{0,1}(\mathrm{End}(E))$, $B\in A^{1,0}(\mathrm{End}(E))$ and $\phi\in A^{0,1}(TX)$. Let $\overline D$ be the differential operator defined as $\overline D : = \overline\partial_{E} + \{\partial_{K},\phi\lrcorner\} + \theta + A + B$.  $\overline D_{}^2=0$ holds if and only if  the following two equations hold:
\begin{equation*}
\left\{
\begin{aligned}
& \overline\partial_{\mathrm{End}(E)}(A+B)-\phi\lrcorner F_{d_{K}}+[\theta,A +B]+\{\partial_{K}^{\mathrm{End}(E)},\phi\lrcorner\}\theta+\{\partial_{K}^{\mathrm{End}(E)},\phi\lrcorner\}(A+B)+\frac{1}{2}[A+B,A+B]=0,\\
&\overline\partial_{TX}\phi-\frac{1}{2}[\phi,\phi]_{SH}=0.\\
\end{aligned}
\right.
\end{equation*}

\end{proposition}
From now on we denote $[\cdot,\cdot]_{SH}$ as $[\cdot,\cdot]$ if there is no confusion.
The proof of the above proposition will be given at the end of the section.\par
Let us define some notation. Let $L^i$ be $L^i:=\oplus_{p+q=i}A^{p,q}(\mathrm{End}(E))\oplus A^{0,i}(TX))$. Let for $\phi\in A^{0,i}(TX),$ $\{\partial_{K}^{\mathrm{End}(E)},\phi\lrcorner\}:=\partial_{K}^{\mathrm{End}(E)}(\phi\lrcorner)+(-1)^i\phi\lrcorner\partial_{K}^{\mathrm{End}(E)}.$ Define the bracket $[\cdot,\cdot]_{T(E)}:L^i\times L^j\to L^{i+j}$ by
\begin{center}
$[(A,\phi),(B,\psi)]_{T(E)}:=((-1)^i\{\partial_{K}^{\mathrm{End}(E)},\psi\lrcorner\}A-(-1)^{(i+1)j}\{\partial_{K}^{\mathrm{End}(E)},\phi\lrcorner\}B-[A,B],[\phi,\psi])$
\end{center}
Let $B\in A^{0,1}(Hom(TX,\mathrm{End}(E)))$ and $C\in A^{1,0}(Hom(TX,\mathrm{End}(E)))$ acts on $v\in A^{0,p}(TX)$ as,
\begin{center}
$B(v):=(-1)^{p}v\lrcorner F_{d_{K}},$ $C(v):=\{\partial^{\mathrm{End}(E)}_{K},v\lrcorner\}\theta.$
\end{center}
We define the linear operator $d_{T(E)}:L\to L$ as,
\begin{center}
$d_{T(E)} :=
\begin{pmatrix}
\overline\partial_{\mathrm{End}(E)} & B\\
0 & \overline\partial_{TX} 
\end{pmatrix}
+
\begin{pmatrix}
\theta & C\\
0 & 0
\end{pmatrix}$
\end{center}

\begin{theorem}\label{DGLA}
$(L=\oplus_i L_i,[\cdot,\cdot]_{T(E)},d_{T(E)})$ is a DGLA.
\end{theorem}

We separate the proof of the theorem into the two propositions below.
Before going to the proof, we introduce some formulas which are useful for the proof.

\begin{lemma}[{\cite[Lemma 3.1.]{M2}}]\label{Cartan}
Let $i_\xi(\omega)=\xi\lrcorner\omega$ for all $\omega\in A^{*}(X).$ For every $\xi,\eta\in A^{0,*}(TX)$, the following equations holds:
\begin{equation}\label{Cartan a}
i_{[\xi,\eta]}=[i_\xi,[\partial,i_\eta]],~ [i_\xi,i_\eta]=0.
\end{equation}
\end{lemma}
We slightly modify the Lemma \ref{Cartan} so that we can use it in our proof.
\begin{lemma}\label{Cartan b}
Let  $X$  be a complex manifold and $E$ be a holomorphic bundle over $X$.  Let $K$ be a hermitian metric on $E$ and $\partial_K$ be a (1,0)-part of the Chern connection. By considering the degree of the differential form of (\ref{Cartan a}), for any $\omega\in A^{*}(E)$ and any $\phi\in A^{0,j}(TX)$ and $\psi\in A^{0,k}(TX),$ we have 
\begin{equation}
[\phi,\psi]\lrcorner\omega=\phi\lrcorner\partial_K(\psi\lrcorner\omega)-(-1)^{jk+k}\partial_K(\psi\lrcorner(\phi\lrcorner\omega))-(-1)^{jk}\psi\lrcorner\partial_K(\phi\lrcorner\omega)-(-1)^{jk+k}\psi\lrcorner\phi\lrcorner\partial_K\omega.
\end{equation}
\end{lemma}

We obtain the corollaries below by applying lemma \ref {Cartan b}.
\begin{corollary}\label{Cartan b vb}
Let $A\in A^i(\mathrm{End}(E))$, $\phi\in A^{0,j}(TX)$ and $\psi\in A^{0,k}(TX)$. Then we have,
\begin{equation*}\label{Cartan bracket}
\{\partial_{K}^{\mathrm{End}(E)},[\phi,\psi]\lrcorner\}A=\{\partial_{K}^{\mathrm{End}(E)}, \phi\lrcorner\}\{\partial_{K}^{\mathrm{End}(E)}, \psi\lrcorner\}A-(-1)^{jk}\{\partial_{K}^{\mathrm{End}(E)},\psi\lrcorner\}\{\partial_{K}^{\mathrm{End}(E)}, \phi\lrcorner\}A.
\end{equation*}
\end{corollary}
\begin{proof}
We denote $\partial_{K}^{\mathrm{End}(E)}$ as $\partial_K$ in this proof.\par
Applying lemma \ref {Cartan b} to the left-hand side of the equation we have,
\begin{align*}
\{\partial_K,[\phi,\psi]\lrcorner\}A=& \partial_K([\phi,\psi]\lrcorner A)+(-1)^{j+k}[\phi,\psi]\lrcorner\partial_KA\\
=& \partial_K\{\phi\lrcorner\partial_K(\psi\lrcorner A)-(-1)^{jk+k}\partial_K(\psi\lrcorner(\phi\lrcorner A))-(-1)^{jk}\psi\lrcorner\partial_K(\phi\lrcorner A)-(-1)^{jk+k}\psi\lrcorner\phi\lrcorner\partial_K A\}\\
&+(-1)^{j+k}\{\phi\lrcorner\partial_K(\psi\lrcorner\partial_K A)-(-1)^{jk+k}\partial_K(\psi\lrcorner(\phi\lrcorner\partial_K A))-(-1)^{jk}\psi\lrcorner\partial_K(\phi\lrcorner\partial_K A)\}\\
=&\partial_K(\phi\lrcorner\partial_K(\psi\lrcorner A))-(-1)^{jk}\partial_K(\psi\lrcorner\partial_K(\phi\lrcorner A))-(-1)^{jk+k}\partial_K(\psi\lrcorner\phi\lrcorner\partial_K A)\\
&+(-1)^{j+k}\{\phi\lrcorner\partial_K(\psi\lrcorner\partial_K A)-(-1)^{jk+k}\partial_K(\psi\lrcorner(\phi\lrcorner\partial_K A))-(-1)^{jk}\psi\lrcorner\partial_K(\phi\lrcorner\partial_K A)\}\\
=&\partial_K(\phi\lrcorner\partial_K(\psi\lrcorner A))-(-1)^{jk}\partial_K(\psi\lrcorner\partial_K(\phi\lrcorner A))-(-1)^{jk+k}\partial_K(\psi\lrcorner\phi\lrcorner\partial_K A)\\
&+(-1)^{j+k}\phi\lrcorner\partial_K(\psi\lrcorner\partial_K A)-(-1)^{jk+j}\partial_K(\psi\lrcorner(\phi\lrcorner\partial_K A))-(-1)^{jk+j+k}\psi\lrcorner\partial_K(\phi\lrcorner\partial_K A)\}.
\end{align*}
We apply lemma \ref{Cartan a} for the computation of the right-hand side of the equation.
\begin{align*}
&\{\partial_K, \phi\lrcorner\}\{\partial_K, \psi\lrcorner\}A-(-1)^{jk}\{\partial_K, \psi\lrcorner\}\{\partial_K, \phi\lrcorner\}A\\
=&\{\partial_K, \phi\lrcorner\}(\partial_K\psi\lrcorner A+(-1)^k\psi\lrcorner\partial_K )A-(-1)^{jk}\{\partial_K, \psi\lrcorner\}(\partial_K\phi\lrcorner A+(-1)^j\phi\lrcorner\partial_K A)\\
=&\partial_K(\phi\lrcorner\partial_K(\psi\lrcorner A))+(-1)^k\partial_K(\phi\lrcorner\psi\lrcorner\partial_K A)+(-1)^{j+k}\phi\lrcorner\partial_K(\psi\lrcorner\partial_K A)\\
&-(-1)^{jk}\{\partial_K(\psi\lrcorner\partial_K(\phi\lrcorner A))+(-1)^j\partial_K(\psi\lrcorner\phi\lrcorner\partial_K A)+(-1)^{j+k}\psi\lrcorner\partial_K(\phi\lrcorner\partial_K A)\}\\
=&\partial_K(\phi\lrcorner\partial_K(\psi\lrcorner A))-(-1)^{jk+k}\partial_K(\psi\lrcorner\phi\lrcorner\partial_K A)+(-1)^{j+k}\phi\lrcorner\partial_K(\psi\lrcorner\partial_K A)\\
&-(-1)^{jk}\partial_K(\psi\lrcorner\partial_K(\phi\lrcorner A))-(-1)^{jk+j}\partial_K(\psi\lrcorner\phi\lrcorner\partial_K A)-(-1)^{jk+j+k}\psi\lrcorner\partial_K(\phi\lrcorner\partial_K A).
\end{align*}
Hence we have equality holds.
\end{proof}

\begin{corollary}\label{Cartan curvature}
Let $F_{d_K}$ be the curvature of the Chern connection. Let $\phi\in A^{0,i}(TX)$ and $\psi\in A^{0,j}(TX)$. Then we have
\begin{equation*}
[\phi,\psi]\lrcorner F_{d_K}=(-1)^i\{\partial_{K}^{\mathrm{End}(E)},\phi\lrcorner\}\psi\lrcorner F_{d_K}-(-1)^{ij+j}\{\partial_{K}^{\mathrm{End}(E)},\psi\lrcorner\}\phi\lrcorner F_{d_K}.
\end{equation*}
\end{corollary}
\begin{proof}
We denote $\partial_{K}^{\mathrm{End}(E)}$ as $\partial_K$ in this proof.\par
Recall that  $F_{d_K}$ is a $(1,1)$-form which takes value in $\mathrm{End}(E)$.\par
Applying Lemma \ref {Cartan b} to the left-hand side of the equation and by the Bianchi identity we have,
\begin{align*}
[\phi,\psi]\lrcorner F_{d_K}=&\phi\lrcorner\partial_K(\psi\lrcorner  F_{d_K})-(-1)^{ij+j}\partial_K(\psi\lrcorner(\phi\lrcorner  F_{d_K}))-(-1)^{ij}\psi\lrcorner\partial_K(\phi\lrcorner  F_{d_K})-(-1)^{ij+j}\psi\lrcorner\phi\lrcorner\partial_K  F_{d_K}\\
=&\phi\lrcorner\partial_K(\psi\lrcorner  F_{d_K})-(-1)^{ij}\psi\lrcorner\partial_K(\phi\lrcorner  F_{d_K}).
\end{align*}
By direct computation for the right-hand side of the equation, we have,
\begin{equation*}
\begin{split}
&(-1)^i\{\partial_K,\phi\lrcorner\}\psi\lrcorner F_{d_K}-(-1)^{ij+j}\{\partial_K,\psi\lrcorner\}\phi\lrcorner F_{d_K}\\
=&(-1)^i\partial_K(\phi\lrcorner\psi\lrcorner F_{d_K})+ \phi\lrcorner\partial_K\psi\lrcorner F_{d_K}-(-1)^{ij+j}\partial_K(\psi\lrcorner\phi\lrcorner F_{d_K})-(-1)^{ij} \psi\lrcorner\partial_K\phi\lrcorner F_{d_K}\\
=&\phi\lrcorner\partial_K\psi\lrcorner F_{d_K}-(-1)^{ij} \psi\lrcorner\partial_K\phi\lrcorner F_{d_K}.
\end{split}
\end{equation*}
Hence we have the desired equality.
\end{proof}
By direct computation, we obtain some corollaries.
\begin{corollary}\label{Leibniz a}
Let $A\in A^i(\mathrm{End}(E))$, $B\in A^j(\mathrm{End}(E))$ and $\phi\in A^{0,k}(TX)$. Then we have 
\begin{equation}
\{\partial_{K}^{\mathrm{End}(E)},\phi\lrcorner\}[A,B]=[\{\partial_{K}^{\mathrm{End}(E)},\phi\lrcorner\}A,B]+(-1)^{ik}[A,\{\partial_{K}^{\mathrm{End}(E)},\phi\lrcorner\}B].
\end{equation}
\end{corollary} 
\begin{proof}
We denote $\partial_{K}^{\mathrm{End}(E)}$ as $\partial_K$ in this proof.\par
By using local trivialization we have,
\begin{align*}
\{\partial_K,\phi\lrcorner\}[A,B]&=\partial_K(\phi\lrcorner[A,B])+(-1)^k\phi\lrcorner\partial_K[A,B]\\
=&\partial(\phi\lrcorner[A,B])+[K^{-1}\partial K,\phi\lrcorner[A,B]]+(-1)^{k}\phi\lrcorner(\partial[A,B])+[K^{-1}\partial K,[A,B]])\\
=&[\partial(\phi\lrcorner A),B]+(-1)^{i+k-1}[\phi\lrcorner A,\partial B]+(-1)^{i+ik}[\partial A,\phi\lrcorner B]+(-1)^{ik}[A,\partial(\phi\lrcorner B)]\\
&+(-1)^k[\phi\lrcorner\partial A,B]+(-1)^{ik+i+1}[\partial A,\phi\lrcorner B]+(-1)^{k+i}[\phi\lrcorner A,\partial B]+(-1)^{k+ki}[A,\phi\lrcorner\partial B]\\
&+(-1)^k[\phi\lrcorner K^{-1}\partial K,[A,B]]\\
=&[\partial(\phi\lrcorner A),B]+(-1)^{ik}[A,\partial(\phi\lrcorner B)]+(-1)^k[\phi\lrcorner\partial A,B]+(-1)^{ik+k}[A,\phi\lrcorner\partial B]\\
&+(-1)^k[[\phi\lrcorner K^{-1}\partial K,A],B]+(-1)^{ki+k}[A,[\phi\lrcorner K^{-1}\partial K,B]]\\
=&[\{\partial_K,\phi\lrcorner\}A,B]+(-1)^{ik}[A,\{\partial_K,\phi\lrcorner\}B].
\end{align*}
Hence we have the desired equality.
\end{proof}
\begin{corollary}\label{Leibniz b}
Let $A\in A^{i}(\mathrm{End}(E))$ and $\phi\in A^{0,j}(TX).$ Then the following equality holds:
\begin{equation*}
\overline\partial_{\mathrm{End}(E)}\{\partial_{K}^{\mathrm{End}(E)},\phi\lrcorner\}A=(-1)^j\{\partial_{K}^{\mathrm{End}(E)},\phi\lrcorner\}\overline\partial_{\mathrm{End}(E)}A-\{\partial_{K}^{\mathrm{End}(E)},\overline\partial_{TX}\phi\lrcorner\}A-[\phi\lrcorner F_{d_K},A].
\end{equation*}
\end{corollary}
\begin{proof}
We denote $\partial_{K}^{\mathrm{End}(E)}$ as $\partial_K$ in this proof.\par
We prove the above equality by using local trivialization.
\begin{equation*}
\begin{split}
\overline\partial_{\mathrm{End}(E)}\{\partial_K,\phi\lrcorner\}A=&\overline\partial_{\mathrm{End}(E)}\{\partial_K(\phi\lrcorner)+(-1)^j\phi\lrcorner\partial_K A\}\\
=&\overline\partial_{\mathrm{End}(E)}\{\partial(\phi\lrcorner A)+(-1)^j[\phi\lrcorner K^{-1}\partial K,A]+(-1)^j\phi\lrcorner\partial A\}\\
=&-\partial(\overline\partial_{TX}\phi\lrcorner A)+(-1)^j\partial(\phi\lrcorner\overline\partial_{\mathrm{End}(E)}A)+(-1)^j[\overline\partial_{EndE}(\phi\lrcorner K^{-1}\partial K ),A]\\
&+[\phi\lrcorner K^{-1}\partial K ,\overline\partial_{\mathrm{End}(E)}A]
+(-1)^j\overline\partial_{TX}\phi\lrcorner\partial A-\phi\lrcorner\overline\partial_{\mathrm{End}(E)}\partial A\\
=&(-1)^j\{\partial_K,\phi\lrcorner\}\overline\partial_{\mathrm{End}(E)}A-[\phi\lrcorner F_{d_K},A]-\partial(\overline\partial_{TX})\phi\lrcorner A+(-1)^j[(\overline\partial_{TX}\phi)\lrcorner K^{-1}\partial K,A]+(-1)^j\overline\partial_{TX}\phi\lrcorner\partial A\\
=&(-1)^j\{\partial_K,\phi\lrcorner\}\overline\partial_{\mathrm{End}(E)}A-\{\partial_K,\overline\partial_{TX}\phi\lrcorner\}A-[\phi\lrcorner F_{d_K},A].
\end{split}
\end{equation*}
Hence we have the desired equality.
\end{proof}

\begin{proposition} \label{bracket for DGLA}
The bracket $[\cdot,\cdot]_{T(E)}:L\times L\to L$ satisfies the following,
\begin{enumerate}
\item $ (A,\phi)\in L^i, (B,\psi)\in L^j (i, j\in\mathbb{Z})$, $[(A,\phi),(B,\psi)]_{T(E)}+(-1)^{pq}[(B,\psi),(A,\phi)]_{T(E)}=0$.
\item The graded Jacobi identity holds.
\end{enumerate}
\end{proposition}
\begin{proof}
We denote $\partial_{K}^{\mathrm{End}(E)}$ as $\partial_K$ in this proof.\par
$\mathit{1.}$ is obvious form the definition. We only prove $\mathit{2}.$\par
Let $(A,\phi)\in L^i$, $(B,\psi)\in L^j$ and $(C,\tau)\in L^k.$ We prove the following eqaution:
\begin{equation}
[(A,\phi),[(B,\psi),(C,\tau)]_{T(E)}]_{T(E)}=[[(A,\phi),(B,\psi)]_{T(E)},(C,\tau)]_{T(E)}+(-1)^{ij}[(B,\psi),[(A,\phi),(C,\tau)]_{T(E)}]_{T(E)}.
\end{equation}
We have,
\begin{equation}\label{bracket calculation}
\begin{split}
&[(A,\phi),[(B,\psi),(C,\tau)]_{T(E)}]_{T(E)}\\
=&[(A,\phi),((-1)^j\{\partial_K,\psi\lrcorner\}C-(-1)^{(j+1)k}\{\partial_K,\tau\lrcorner\}B-[B,C],[\psi,\tau])]_{T(E)}\\
=&((-1)^i\{\partial_K,\phi\lrcorner\}\{(-1)^j\{\partial_K,\psi\lrcorner\}C-(-1)^{(j+1)k}\{\partial_K,\tau\lrcorner\}B-[B,C]\}-(-1)^{(i+1)(j+k)}\{\partial_K,[\psi,\tau]\lrcorner\}A\\
&-[A,(-1)^j\{\partial_K,\psi\lrcorner\}C-(-1)^{(j+1)k}\{\partial_K,\tau\lrcorner\}B-[B,C]],[\phi,[\psi,\tau]]).\\
\textrm{and}&\\
&[[(A,\phi),(B,\psi)]_{T(E)},(C,\tau)]_{T(E)}\\
=&((-1)^{i+j}\{\partial,[\phi,\psi]\lrcorner\}C-(-1)^{(i+j+1)k})\{\partial_K,\tau\lrcorner\}\{(-1)^i\{\partial_K,\phi\lrcorner\}B-(-1)^{(i+1)j}\{\partial_K,\psi\lrcorner\}A-[A,B]\}\\
&-[(-1)^i\{\partial_K,\phi\lrcorner\}B-(-1)^{(i+1)j}\{\partial_K,\psi\lrcorner\}A-[A,B],C],[[\phi,\psi],\tau]).\\
&(-1)^{ij}[(B,\psi),[(A,\phi),(C,\tau)]_{T(E)}]_{T(E)}\\
\textrm{and}&\\
=&(-1)^{ij}((-1)^i\{\partial_K,\psi\lrcorner\}((-1)^i\{\partial_K,\phi\lrcorner\}C-(-1)^{(i+1)k}\{\partial_K,\phi\lrcorner\}A-[A,C]-(-1)^{(j+1)(i+k)}\{\partial_K,[\phi,\tau]\lrcorner\}B\\
&-[B,(-1)^i\{\partial_K,\psi\lrcorner\}((-1)^i\{\partial_K,\phi\lrcorner\}C-(-1)^{(i+1)k}\{\partial_K,\phi\lrcorner\}A-[A,C]],[\psi,[\phi,\tau]]).
\end{split}
\end{equation}
Hence by \eqref{bracket calculation}, we only have to prove the following equations,
\begin{equation*}
\begin{split}
\{\partial_K,[\phi,\psi]\lrcorner\}A&=\{\partial_K, \phi\lrcorner\}\{\partial_K, \psi\lrcorner\}A-(-1)^{jk}\{\partial_K, \psi\lrcorner\}\{\partial_K, \phi\lrcorner\}A,\\
\{\partial_K,\phi\lrcorner\}[A,B]&=[\{\partial_K,\phi\lrcorner\}A,B]+(-1)^{ik}[A,\{\partial_K,\phi\lrcorner\}B],\\
[A,[B,C]]&=[[A,B],C]+(-1)^{ij}[B,[A,C]],\\
[\phi,[\psi,\tau]]&=[[\phi,\psi],\tau]+(-1)^{ij}[\psi,[\phi,\tau]].
\end{split}
\end{equation*}
The above equations follow from corollary \ref{Cartan b vb} and  \ref{Leibniz a} and the fact that the Schouten-Nijenhuis bracket satisfies the Jacobi identity. Hence we proved that $[\cdot,\cdot]_{T(E)}$ satisfies the Jacobi identity.
\end{proof}

\begin{proposition}\label{d for DGLA}
$d_{T(E)}$ is a differential respect to the bracket $[\cdot,\cdot]_{T(E)}$, that is :
\begin{enumerate}
\item $d_{T(E)}(L^{i})\subset L^{i+1}.$
\item $d_{T(E)}\circ d_{T(E)}= 0.$
\item $(A,\phi)\in L^i, (B,\psi)\in L^j, d_{T(E)}[(A,\phi),(B,\psi)]_{T(E)}=[d_{T(E)}(A,\phi),(B,\psi)]_{T(E)}+(-1)^i[(A,\phi),d_{T(E)}(B,\psi)]_{T(E)}.$
\end{enumerate}
\end{proposition}
\begin{proof}
We denote $\partial_{K}^{\mathrm{End}(E)}$ as $\partial_K$ in this proof.\par
$\mathit{1.}$ is obvious from the definition of $d_{T(E)}$.\par
We prove $\mathit{2.}$ for $d_{T(E)}\circ d_{T(E)}:L^1\to L^3.$ Let $(A,\phi)\in L^1$. We calculate $d_{T(E)}\circ d_{T(E)}(A,\phi).$ 
\begin{align*}
d_{T(E)}(A,\phi)=
\begin{pmatrix}
\overline\partial_{\mathrm{End}(E)}A-\phi\lrcorner F_{d_K}\\
\overline\partial_{T(X)}
\end{pmatrix}
+
\begin{pmatrix}
[\theta,A] + \{\partial_K,\phi\lrcorner\}\theta\\
0
\end{pmatrix}.
\end{align*}
\begin{align}
\notag d_{T(E)}&\bigg(
\begin{pmatrix}
\overline\partial_{\mathrm{End}(E)}A-\phi\lrcorner F_{d_K}\\
\overline\partial_{T(X)}\phi
\end{pmatrix}
+
\begin{pmatrix}
[\theta,A] + \{\partial_K,\phi\lrcorner\}\theta\\
0
\end{pmatrix}
\bigg)\\
=&
\begin{pmatrix}
\overline\partial_{\mathrm{End}(E)}&B\\
0 & \overline\partial_{T(X)}
\end{pmatrix}
\begin{pmatrix}
\overline\partial_{\mathrm{End}(E)}A-\phi\lrcorner F_{d_K}\\
\overline\partial_{T(X)}\phi
\end{pmatrix}
\label{A}
\\
&+
\begin{pmatrix}
\theta & C\\
0 & 0
\end{pmatrix}
\begin{pmatrix}
\overline\partial_{\mathrm{End}(E)}A-\phi\lrcorner F_{d_K}\\
\overline\partial_{TX}\phi
\end{pmatrix}
+
\begin{pmatrix}
\overline\partial_{\mathrm{End}(E)}&B\\
0 & \overline\partial_{TX}
\end{pmatrix}
\begin{pmatrix}
[\theta,A] + \{\partial_K,\phi\lrcorner\}\theta\\
0
\end{pmatrix}
\label{B}
\\
&+
\begin{pmatrix}
\theta & C\\
0 & 0
\end{pmatrix}
\begin{pmatrix}
[\theta,A] + \{\partial_K,\phi\lrcorner\}\theta\\
0
\end{pmatrix}
\label{C}.
\end{align}
Let us show \eqref{A}=\eqref{B}=\eqref{C}=0
\begin{align*}
\eqref{A}&=
\begin{pmatrix}
\overline\partial_{\mathrm{End}(E)}&B\\
0 & \overline\partial_{TX}
\end{pmatrix}
\begin{pmatrix}
\overline\partial_{\mathrm{End}(E)}A-\phi\lrcorner F_{d_K}\\
\overline\partial_{TX}\phi
\end{pmatrix}
=
\begin{pmatrix}
\overline\partial_{\mathrm{End}(E)}\circ\overline\partial_{\mathrm{End}(E)}A + \overline\partial_{\mathrm{End}(E)}(\phi\lrcorner F_{d_K}) + B(\overline\partial_{T(X)}\phi)\\
\overline\partial_{TX}\circ\overline\partial_{TX}\phi
\end{pmatrix}\\
&=
\begin{pmatrix}
\overline\partial_{TX}\phi\lrcorner F_{d_K} + \phi\lrcorner\overline\partial_{\mathrm{End}(E)}F_{d_k}-\overline\partial_{TX}\phi\lrcorner F_{d_K}\\
0
\end{pmatrix}
=
\begin{pmatrix}
 \phi\lrcorner\overline\partial_{\mathrm{End}(E)}F_{d_k}\\
0
\end{pmatrix}=0. 
\end{align*}
The last equation comes from Bianchi identity. Next we show that \eqref{B}=0
\begin{align*}
\eqref{B}&=
\begin{pmatrix}
\theta & C\\
0 & 0
\end{pmatrix}
\begin{pmatrix}
\overline\partial_{\mathrm{End}(E)}A-\phi\lrcorner F_{d_K}\\
\overline\partial_{TX}\phi
\end{pmatrix}
+
\begin{pmatrix}
\overline\partial_{\mathrm{End}(E)}&B\\
0 & \overline\partial_{TX}
\end{pmatrix}
\begin{pmatrix}
[\theta,A] + \{\partial_K,\phi\lrcorner\}\theta\\
0
\end{pmatrix}\\
&=
\begin{pmatrix}
[\theta,\overline\partial_{\mathrm{End}(E)} A]-[\theta,\phi\lrcorner F_{d_K}]+\{\partial_{K},\overline\partial_{TX}\phi\lrcorner\}\theta\\
0
\end{pmatrix}
+
\begin{pmatrix}
\overline\partial_{\mathrm{End}(E)}[\theta,A] + \overline\partial_{\mathrm{End}(E)}(\{\partial_K,\phi\lrcorner\}\theta)\\
0
\end{pmatrix}.
\end{align*}
Since $\theta$ is a Higgs field, $\overline\partial_{\mathrm{End}(E)}[\theta,A]=-[\theta,\overline\partial_{\mathrm{End}(E)}A]$.  Hence we have
\begin{equation}
\eqref{B}=
\begin{pmatrix}
-[\theta,\phi\lrcorner F_{d_K}]+\{\partial_{K},\overline\partial_{TX}\phi\lrcorner\}\theta+\overline\partial_{\mathrm{End}(E)}(\{\partial_K,\phi\lrcorner\}\theta)\\
0
\end{pmatrix}.\label{D}
\end{equation}
By direct computation using the local realization we have,
\begin{align}
\{\partial_{K},\overline\partial_{TX}\phi\lrcorner\}\theta&=\notag\partial_K (\overline\partial_{TX}\phi\lrcorner\theta)+\overline\partial_{TX}\phi\lrcorner(\partial_K \theta)\\
&=\partial (\overline\partial_{TX}\phi\lrcorner\theta)+[K^{-1}\partial K, \overline\partial_{TX}\phi\lrcorner\theta]+\overline\partial_{TX}\phi\lrcorner(\partial\theta + [K^{-1}\partial K, \theta]).\label{E}
\end{align}
\begin{align}
\overline\partial_{\mathrm{End}(E)}(\{\partial_K,\phi\lrcorner\}\theta)&=\notag\overline\partial_{\mathrm{End}(E)}\{\partial(\phi\lrcorner\theta)+[K^{-1}\partial K,\phi\lrcorner\theta]-\phi\lrcorner(\partial\theta + [K^{-1}\partial K,\theta])\}\\
&\notag=-\partial\overline\partial(\phi\lrcorner\theta)+[F_{d_{K}},\phi\lrcorner\theta]-[K^{-1}\partial K,\overline\partial_{\mathrm{End}(E)}(\phi\lrcorner\theta)]-\overline\partial_{TX}\phi\lrcorner(\partial\theta+[K^{-1}\partial K,\theta])-\phi\lrcorner[F_{d_{K}},\theta]\\
&=-\partial(\overline\partial_{TX}\phi\lrcorner\theta)-[\phi\lrcorner F_{d_{K}},\theta]-[K^{-1}\partial K, \overline\partial_{TX}\phi\lrcorner\theta]-\overline\partial_{TX}\phi\lrcorner(\partial\theta+[K^{-1}\partial K,\theta])\label{F}.
\end{align}
Hence by \eqref{D}, \eqref{E}, \eqref{F} we obtain that \eqref{B}=0.\par
Next, we show \eqref{C}=0. 
\begin{align}
\eqref{C}=
\begin{pmatrix}
\theta & C\\
0 & 0
\end{pmatrix}
\begin{pmatrix}
[\theta, A]+\{\partial_K,\phi\lrcorner\}\theta\\
0
\end{pmatrix}
=
\begin{pmatrix}
[\theta,\{\partial_K,\phi\lrcorner\}\theta]\\
0
\end{pmatrix}\label{G}.
\end{align}
 By direct computation using the local realization we have,
 \begin{align}
 [\theta,\{\partial_K,\phi\lrcorner\}\theta]=&\notag\theta\wedge\{\partial_K,\phi\lrcorner\}\theta-\{\partial_K,\phi\lrcorner\}\theta\wedge\theta\\
 \notag=&\theta\wedge\{\partial(\phi\lrcorner\theta)+[K^{-1}\partial K,\phi\lrcorner\theta]-\phi\lrcorner\partial\theta-\phi\lrcorner[K^{-1}\partial K,\theta]\}\\
 &\notag-\{\partial(\phi\lrcorner\theta)+[K^{-1}\partial K,\phi\lrcorner]\theta-\phi\lrcorner\partial\theta-\phi\lrcorner[K^{-1}\partial K,\theta]\}\wedge\theta\\
 \notag=&\theta\wedge\{\partial(\phi\lrcorner\theta)-\phi\lrcorner\partial\theta-[\phi\lrcorner K^{-1}\partial K,\theta]\}-\{\partial(\phi\lrcorner\theta)-\phi\lrcorner\partial\theta-[\phi\lrcorner K^{-1}\partial K,\theta]\}\wedge\theta\\
 =&\theta\wedge\partial(\phi\lrcorner\theta)-\theta\wedge\phi\lrcorner\partial\theta-\partial(\phi\lrcorner\theta)\wedge\theta+(\phi\lrcorner\partial\theta)\wedge\theta\label{H}.
  \end{align}
 Since $\theta\wedge\theta=0$,
 \begin{align}
 0=&\notag\partial(\phi\lrcorner(\theta\wedge\theta))-\phi(\partial(\theta\wedge\theta))\\
 =&\notag\partial(\phi\lrcorner\theta)\wedge\theta-(\phi\lrcorner\theta)\wedge\partial\theta+\partial\theta\wedge\phi\lrcorner\theta-\theta\wedge\partial(\phi\lrcorner\theta)-(\phi\lrcorner\partial\theta)\wedge\theta-\partial\theta\wedge(\phi\lrcorner\theta)+(\phi\lrcorner\theta)\wedge\partial\theta+\theta\wedge(\phi\lrcorner\partial\theta)\\
 =&\partial(\phi\lrcorner\theta)\wedge\theta-\theta\wedge\partial(\phi\lrcorner\theta)-(\phi\lrcorner\partial\theta)\wedge\theta+\theta\wedge\phi\lrcorner\partial\theta\label{I}.
    \end{align}
Hence by \eqref{G}, \eqref{H} and \eqref{I} we obtain that \eqref{C}=0. This completes the proof of $\mathit{2}.$\par
Next we prove $\mathit{3}.$\par
Let $(A,\phi)\in L^i$ and $(B,\psi)\in L^j.$    We have
\begin{equation*}
\begin{split}
&d_{T(E)}[(A,\phi),(B,\psi)]_{T(E)}\\
=&
\begin{pmatrix}
(\overline\partial_{\mathrm{End}(E)}+\theta)((-1)^i\{\{\partial_K,\phi\lrcorner\}B-(-1)^{(i+1)j}\{\partial_K,\psi\lrcorner\}A-[A,B]\})+(-1)^{i+j}[\phi,\psi]\lrcorner F_{d_K}+\{\partial_K,[\phi,\psi]\lrcorner\}\theta\\
\overline\partial_{TX}[\phi,\psi]
\end{pmatrix}
\end{split}
\end{equation*} 
and
\begin{equation*}
\begin{split}
&[d_{T(E)}(A,\phi),(B,\psi)]_{T(E)}\\
=&
((-1)^{i+1}\{\partial_K,\overline\partial_{TX}\phi\lrcorner\}B-(-1)^{(i+2)j}\{\partial_K,\psi\lrcorner\}(\overline\partial_{\mathrm{End}(E)}A+(-1)^i\phi\lrcorner F_{d_K}+[\theta,A]+\{\partial_K,\phi\lrcorner\}\theta\\
&-[\overline\partial_{\mathrm{End}(E)}A+(-1)^i\phi\lrcorner F_{d_K}+[\theta,A]+\{\partial_K,\phi\lrcorner\}\theta,B]
,[\overline\partial_{TX}\phi,\psi])\\
\textrm{and}&\\
&(-1)^i[(A,\phi),d_{T(E)}(B,\psi)]_{T(E)}\\
=&(\{\partial_K,\phi\lrcorner\}(\overline\partial_{\mathrm{End}(E)}B+(-1)^{j}\psi\lrcorner F_{d_K}+[\theta,B]+\{\partial_K,\phi\lrcorner\}\theta)\\
&-(-1)^{(i+1)(j+1)+i}\{\partial_K,\overline\partial_{TX}\psi\lrcorner\}A-(-1)^i[A,\overline\partial_{\mathrm{End}(E)}B+(-1)^{j}\psi\lrcorner F_{d_K}+[\theta,B]+\{\partial_K,\psi\lrcorner\}\theta],(-1)^i[\phi,\overline\partial_{TX}\psi]).
\end{split}
\end{equation*}
Hence by the above equations, we have to prove the following,
\begin{equation*}
\begin{split}
\overline\partial_{\mathrm{End}(E)}\{\partial_K,\phi\lrcorner\}A&=(-1)^j\{\partial_K,\phi\}\overline\partial_{\mathrm{End}(E)}A-\{\partial_K,\overline\partial_{TX}\phi\lrcorner\}A-[\phi\lrcorner F_{d_K},A],\\
\{\partial_K,\phi\lrcorner\}[\theta,A]&=[\{\partial_K,\phi\lrcorner\}\theta,A]+(-1)^{i}[\theta,\{\partial_K,\phi\lrcorner\}A],\\
[\phi,\psi]\lrcorner F_{d_K}&=(-1)^i\{\partial_K,\phi\lrcorner\}\psi\lrcorner F_{d_K}-(-1)^{ij+j}\{\partial_K,\psi\lrcorner\}\phi\lrcorner F_{d_K}\\
[\theta,[A,B]]&=[[\theta,A],B]+(-1)^i[A,[\theta,B]],\\
\overline\partial_{\mathrm{End}(E)}[A,B]&=[\overline\partial_{\mathrm{End}(E)}A,B]+(-1)^i[A,\overline\partial_{\mathrm{End}(E)}B],\\
\overline\partial_{TX}[\phi,\psi]&=[\overline\partial_{TX}\phi,\psi]+(-1)^i[\phi,\overline\partial_{TX}\psi].
\end{split}
\end{equation*}
These equations follow from the Corollary \ref{Cartan b vb}, \ref{Cartan curvature}, \ref{Leibniz a}, \ref{Leibniz b} and the fact that $\overline\partial_{\mathrm{End}(E)}$ and $\overline\partial_{TX}$ satisfies the Leibniz rule and the canonical bracket satisfies the Jacobi identity. 
 \end{proof}
Proposition \ref{bracket for DGLA} and  \ref{d for DGLA} shows us that $(L,[\cdot,\cdot]_{T(E)},d_{T(E)})$ is a DGLA. Hence we proved Theorem 3.1.
Combining theorem \ref{N-N}, Proposition \ref{vanish a}, Theorem \ref{DGLA}, we have,
\begin{theorem}\label{DGLA for Higgs bundle}
Given a holomorphic-Higgs pair $(X,E,\theta)$ and a smooth family of elements $\{A_{t},B_{t},\phi_{t}\}_{t\in\Delta}\subset A^{0,1}(\mathrm{End}(E))\oplus A^{1,0}(\mathrm{End}(E))\oplus A^{0,1}(TX)$.
Then, $(A_{t},B_{t},\phi_{t})$ defines a holomorphic-Higgs pair if and only if $(A_{t},B_{t},\phi_{t})$ satisfies the Maurer-Cartan:
\begin{align}\label{eq DGLA for Higgs bundle}
d_{T(E)}(A_{t}+B_{t},\phi_{t})-\dfrac{1}{2}[(A_{t}+B_{t},\phi_{t}),(A_{t}+B_{t},\phi_{t})]=0.
\end{align}
\end{theorem}
We now give the proof of proposition \ref{vanish a}.
\begin{proof}[Proof of Proposition \ref{vanish a}]
We calculate $\overline D^2$. Applying Corollary \ref{Cartan b vb}, \ref{Leibniz b}, we have,
\begin{equation*}
\begin{split}
\overline D^2=&(\overline\partial_{E}+\{\partial_K,\phi\lrcorner\}+A+\theta+B)^2\\
=&\overline\partial_E\overline\partial_E +\overline\partial_E\{\partial_K,\phi\lrcorner\}+\{\partial_K,\phi\lrcorner\}\overline\partial_E+\{\partial_K,\phi\lrcorner\}\{\partial_K,\phi\lrcorner\}+\{\partial_K,\phi\lrcorner\}(A+\theta+B) +(A+\theta+B)\{\partial_K,\phi\lrcorner\}\\
&+\overline\partial_E(A+\theta+B)+(A+\theta+B)\overline\partial_E+[\theta,A+B]+\dfrac{1}{2}[A+B,A+B]\\
=&-\{\partial_K,\overline\partial_{TX}\phi\lrcorner\}-\phi\lrcorner F_{d_K}+\dfrac{1}{2}\{\partial_K,[\phi,\phi]\lrcorner\}+\overline\partial_{\mathrm{End}(E)}(\theta+A+B)+\{\partial_{K}^{\mathrm{End}(E)},\phi\lrcorner\}(\theta+A+B)\\
&+[\theta,A+B]+\dfrac{1}{2}[A+B,A+B]\\
=&-\{\partial_K,(\overline\partial_{TX}\phi-\dfrac{1}{2}[\phi,\phi])\lrcorner\}\\
&+\overline\partial_{\mathrm{End}(E)}(A+B)-\phi\lrcorner F_{d_K}+\{\partial_{K}^{\mathrm{End}(E)},\phi\lrcorner\}\theta+[\theta,A+B]+\{\partial_{K}^{\mathrm{End}(E)},\phi\lrcorner\}(A+B)+\dfrac{1}{2}[A+B,A+B].
\end{split}
\end{equation*}
Hence by the above calculation, $\overline D^2=0$ is equivalent to the following equations:
\begin{equation*}
\left\{
\begin{aligned}
 &\overline\partial_{\mathrm{End}(E)}(A+B)-\phi\lrcorner F_{d_{K}}+[\theta,A +B]+\{\partial_{K}^{\mathrm{End}(E)},\phi\lrcorner\}\theta+\{\partial_{K}^{\mathrm{End}(E)},\phi\lrcorner\}(A+B)+\frac{1}{2}[A+B,A+B]=0,\\
&\overline\partial_{TX}\phi-\frac{1}{2}[\phi,\phi]=0.\\
\end{aligned}
\right.
\end{equation*}
Hence we have the proof.
\end{proof}

\section{Kuranishi Family and Obstruction}\label{Ku}
\subsection{Kuranishi famliy}\label{Kuranishi family}
In \cite{Ku} Kuranishi constructed a universal family for an arbitrary complex manifold $X$ over a possible singular analytic space. In this section, we want to construct a family of holomorphic-Higgs pairs over a certain singular space which becomes a universal family in this context.\par
Here we recall some differential operators and inequality we need. These are commonly used in the Hodge theory. We choose a hermitian metric $g$ on $X$ and a hermitian metric $K$ on $E$. Using these two metrics, we can define an inner product $(\cdot,\cdot)$ on $L=\oplus_{i}L^i$. We remark that $L^i$ and $L^j$ are orthogonal with respect to this inner product. We first define the formal adjoint of $d_{T(E)}$ with respect to $(\cdot,\cdot)$ by,
\begin{center}
$(d_{T(E)}\alpha,\beta)=(\alpha,d^{*}_{T(E)}\beta)$.
\end{center}
Then the $Laplacian$ $\Delta_{T(E)}$ is defined by,
\begin{center}
$\Delta_{T(E)}=d_{T(E)}\circ d^{*}_{T(E)}+d^{*}_{T(E)}\circ d_{T(E)}$.
\end{center}
This is an elliptic self-adjoint operator. Hence by \cite[Chapter 4, Theorem 4.12]{Wells}, it has a finite dimension kernel $\mathbb{H}^i$. We call the elements of $\mathbb{H}^i$ a \textit{harmonic form}. Let $\widetilde L^{i}$ be a completion of $L^i$ respect to the inner product $(\cdot,\cdot)$, and let $H: L^i\to\mathbb{H}^i$ be the harmonic projection. The Green's operator $G:L^i\to L^i$ is defined by
\begin{center}
$I=H+\Delta_{T(E)}\circ G=H+G\circ\Delta_{T(E)}$
\end{center}
where $I$ is identity for $L^i$. $H$ and $G$ can be extend to the bounded operator $H,G:\widetilde L^i\to\widetilde L^i$. $G$ commutes with $d_{T(E)}$ and $d^{*}_{T(E)}$.
\par
Now let $\{\eta_1,\dots,\eta_n\}\subset \mathbb{H}^1$ be a basis and $\epsilon_1(t):=\sum_{j=1}^{n}t_j\eta_j\in\mathbb{H}^1$. Consider the equation
\begin{center}
$\epsilon(t)=\epsilon_1(t)+\dfrac{1}{2}d_{ T(E)}^*G[\epsilon(t),\epsilon(t)].$
\end{center}
We define H$\ddot{\mathrm{o}}$lder norm $\norm{\cdot}_{k,\alpha}$ as in \cite{Morrow-Kodaira_book}. We also have the following inequality:
\begin{center}
$\norm{d^{*}_{T(E)}\epsilon}_{k,\alpha}\leq C_{1}\norm{\epsilon}_{k+1.\alpha},$\par
$\norm{[\epsilon,\delta]}_{k,\alpha}\leq C_{2}\norm{\epsilon}_{k+1,\alpha}\norm{\delta}_{k+1,\alpha}.$
\end{center}
In \cite{Dou}, Douglis and Nirenberg proved the following a priori estimate:
\begin{center}
$\norm{\epsilon}_{k,\alpha} \leq C_3(\norm{\Delta_{T(E)}\epsilon}_{k-2,\alpha}+\norm{\epsilon}_{0,\alpha}).$
\end{center}
Applying these and following the proof of \cite[Chapter 4, Proposition 2.3]{Morrow-Kodaira_book} one can deduce an estimate for the Green’s operator $G$:
\begin{center}
$\norm{G\epsilon}_{k,\alpha}\leq C_4\norm{\epsilon}_{k-2,\alpha},$
\end{center}
where all $C_i$'s are positive constants which depend only on $k$ and $\alpha$.\par
Then by applying the proof of \cite[Chapter 4, Proposition 2.4]{Morrow-Kodaira_book} or using the implicit function theorem for Banach spaces as in \cite{Ku}, we obtain a unique solution $\epsilon(t)$ which satisfies the equation
\begin{center}
$\epsilon(t)=\epsilon_1(t)+\dfrac{1}{2}d_{T(E)}^*G[\epsilon(t),\epsilon(t)]_{T(E)},$
\end{center}
which is analytic in the variable in $t$. The solution $\epsilon(t)$ is also a smooth section for $L^1$. Indeed,  by applying the Laplacian to the above equation, we get
\begin{center}
$\Delta_{T(E)}\epsilon(t)-\dfrac{1}{2}d_{T(E)}^*[\epsilon(t),\epsilon(t)]_{T(E)}=0.$
\end{center}
Since $\epsilon(t)$ is holomorphic in $t$, we have,
\begin{center}
$\sum_{i,j}\dfrac{\partial^2}{\partial t_i \partial \overline t_j}\epsilon(t)=0.$
\end{center}
Hence we have,
\begin{center}
$(\Delta_{T(E)}+\sum_{i,j}\dfrac{\partial^2}{\partial t_i \partial \overline t_j})\epsilon(t)-\dfrac{1}{2}d_{T(E)}^*[\epsilon(t),\epsilon(t)]_{T(E)}=0.$
\end{center}
Since the operator 
\begin{center}
$\Delta_{T(E)} + \sum_{i,j}\dfrac{\partial^2}{\partial t_i \partial \overline t_j}$
\end{center}
is elliptic, we can say that $\epsilon(t)$ is smooth by elliptic regularity.
\par
From the discussion so far, we have,
\begin{proposition}
Let $\{\eta_1,\dots,\eta_n\}\subset\mathbb{H}^1$ be a basis. Let $t=(t_1,\dots,t_n)\in\mathbb{C}^n$ and $\epsilon_1(t):=\sum_it_i\eta_i.$ For all $|t|<<1$ we have a $\epsilon(t)$ such that $\epsilon(t)$ satisfies the following equation:
\begin{equation*}
\epsilon(t)=\epsilon_1(t)+\dfrac{1}{2}d_{T(E)}^*G[\epsilon(t),\epsilon(t)]_{T(E)}.
\end{equation*}
Moreover, $\epsilon(t)$ is holomorphic with respect to the variable $t$.
\end{proposition}
Following Kuranishi \cite{Ku} we have,
\begin{proposition}
If we take $|t|$ small enough, the solution $\epsilon(t)$ that satisfies
 \begin{center}
 $\epsilon(t)=\epsilon_{1}(t)+\dfrac{1}{2}d_{T(E)}^{*}G[\epsilon(t),\epsilon(t)]_{T(E)}$
\end{center}
solves the Maurer-Cartan equation if and only if $H[\epsilon(t),\epsilon(t)]=0$. Here $H$ is the harmonic projection.
\end{proposition}
\begin{proof}
Suppose the Maurer-Cartan equation holds. Then,
\begin{equation*}
d_{T(E)}\epsilon(t)-\dfrac{1}{2}[\epsilon(t),\epsilon(t)]_{T(E)}=0.
\end{equation*}
Hence we have,
\begin{center}
 $H[\epsilon(t),\epsilon(t)]_{T(E)}=2Hd_{T(E)}\epsilon(t)=0.$
\end{center}
Conversely, suppose that  $H[\epsilon(t),\epsilon(t)]_{T(E)}=0$. We have to show
\begin{center}
$\delta(t):=d_{T(E)}\epsilon(t)-\dfrac{1}{2}[\epsilon(t),\epsilon(t)]_{T(E)}=0.$
\end{center}
Since $\epsilon(t)$ is a solution to 
 \begin{center}
$\epsilon(t)=\epsilon_{1}(t)+\dfrac{1}{2}d_{T(E)}^{*}G[\epsilon(t),\epsilon(t)]_{T(E)}$
\end{center}
and $\epsilon_1(t)$ is $d_{T(E)}$-closed, applying $d_{T(E)}$ we get
\begin{center}
$d_{T(E)}\epsilon(t)=\dfrac{1}{2}d_{T(E)}d^{*}_{T(E)}G[\epsilon(t),\epsilon(t)]_{T(E)}.$
\end{center}
Hence
\begin{center}
$2\delta(t)=d_{T(E)}d_{T(E)}^{*}G[\epsilon(t),\epsilon(t)]_{T(E)}-[\epsilon(t),\epsilon(t)]_{T(E)}.$
\end{center}
By the Hodge decomposition, we can write
\begin{center}
$[\epsilon(t),\epsilon(t)]_{T(E)}=H[\epsilon(t),\epsilon(t)]_{T(E)}+\Delta_{T(E)}G[\epsilon(t),\epsilon(t)]_{T(E)}=\Delta_{T(E)}G[\epsilon(t),\epsilon(t)]_{T(E)}.$
\end{center}
Therefore,
\begin{equation*}
\begin{split}\label{estimate 1}
2\delta(t)&=d_{T(E)}d_{T(E)}^{*}G[\epsilon(t),\epsilon(t)]_{T(E)}-\Delta_{T(E)}G[\epsilon(t),\epsilon(t)]_{T(E)}\\
&=-d_{T(E)}^{*}d_{T(E)}G[\epsilon(t),\epsilon(t)]_{T(E)}\\
&=-d_{T(E)}^{*}Gd_{T(E)}[\epsilon(t),\epsilon(t)]_{T(E)}\\
&=-2d_{T(E)}^{*}G[d_{T(E)}\epsilon(t),\epsilon(t)]_{T(E)}.
\end{split}
\end{equation*}
Hence we get,
\begin{center}
$\delta(t)=-d_{T(E)}^{*}G[d_{T(E)}\epsilon(t),\epsilon(t)]_{T(E)}=-d_{T(E)}^{*}G[\delta(t)+\dfrac{1}{2}[\epsilon(t),\epsilon(t)]_{T(E)},\epsilon(t)]_{T(E)}=-d_{T(E)}^{*}G[\delta(t),\epsilon(t)]_{T(E)}.$
\end{center}
We used the Jacobi identity in the last equality. Using the estimate
\begin{center}
$\norm{[\xi,\eta]}_{k,\alpha}\leq C_{k,\alpha}\norm{\xi}_{k+1,\alpha}\norm{\eta}_{k+1,\alpha},$
\end{center}
we get,
\begin{center}
$\norm{\delta(t)}_{k,\alpha}\leq C_{k,\alpha}\norm{\delta(t)}_{k,\alpha}\norm{\epsilon(t)}_{k,\alpha}.$
\end{center}
If we take $|t|$ small enough such that $C_{k,\alpha}\norm{\epsilon(t)}_{k,\alpha}<1$, we obtain $\delta(t)=0$. This stands for all $|t|$ small enough.
This finishes the proof.
\end{proof}
In the case when $H[\epsilon(t),\epsilon(t)]_{T(E)}=0$ for all $t$ or $\mathbb{H}^2=0$, we have,
\begin{corollary}
Let $n$ be the dimension of $\mathbb{H}^1$.
If $H[\epsilon(t),\epsilon(t)]_{T(E)}=0$ for all $t$, we have a family of deformation of holomorphic-Higgs pair over a small ball $\Delta$ centered at the origin of $\mathbb{C}^{n}$. \end{corollary}
\begin{proof}
If $H[\epsilon(t),\epsilon(t)]=0$ for all $t$, $\epsilon(t)=(A_t + B_t,\phi_t)$ satisfies the Maurer-Cartan equation and so we obtain holomorphic-Higgs family $(X_t, E_t,\theta_t)$. Since $\phi_t$ is holomorphic for variable $t$, applying the Newlander-Nirenberg theorem, we can define a complex structure on $\mathcal{X}:=X\times\Delta$ such that $X_t=\mathcal{X}|_{X\times\{t\}}$. Let $\mathcal{E}:=E\times\Delta$. By applying the linearized Newlander-Nirenberg theorem as in \cite{Mo}, we have a local frame $\{e(x,t)\}$ of $\mathcal{E}$ on $\mathcal{X}$ such that for each $t$, $\{e(x,t)\}\subset \mathrm{ker}(\overline D_{t}')=\mathrm{ker}(\overline\partial_{E_{t}})$ and is holomorphic respect to variable $t$. Let $\sigma:\mathcal{X}\to\mathcal{E}$ be a smooth section and locally trivialized as 
$\sigma(x,t)=\sum_k s^k(x,t)e_k(x,t)$
where  $s^k$ are smooth function on $\mathcal{X}$. We define $\overline\partial_{\mathcal{E}}:A(\mathcal{E})\to A^{0,1}_{\mathcal{X}}(\mathcal{E})$ as follows:
\begin{center}
$\overline\partial_{\mathcal{E}}(\sigma(x,t)):=\sum_k \overline\partial_{\mathcal{X}}s^k(x,t)\otimes e_k(x,t).$
\end{center}
Note that $\overline\partial_{\mathcal{E}}$ is well defined and $\overline\partial_{\mathcal{E}}|_{E\times{t}}=\overline\partial_{E_{t}}$. It is clear that $\overline\partial_{\mathcal{E}}^2=0$ so that $\mathcal{E}$ is a holomorphic bundle over $\mathcal{X}$.\par
Let $\Theta=\theta + B_t +\phi_t\lrcorner(\theta + B_t)$. Since $\phi_t,B_t$ is holomorphic respect to the variable $t$ and $\theta + B_t +\phi_t\lrcorner(\theta + B_t)$ is a Higgs field for $(X_t,E_t)$, we have $\overline\partial_{\mathrm{End}\mathcal{E}}\Theta=0$, $\Theta\wedge\Theta=0$. Hence $\Theta$ is a Higgs field for $(\mathcal{X},\mathcal{E}).$
\par
Let $\pi:\mathcal{X}=X\times\Delta\to\Delta$ be a natural projection, this is a holomorphic submersion. Also $\pi^{-1}(0)=X$, $\mathcal{E}|_{\pi^{-1}(0)}=E$ and $\Theta|_{\pi^{-1}(0)}=\theta$ stands. Hence we have a family of deformation of holomorphic-Higgs pair over $\Delta$.
\end{proof}
In general, the condition $\mathbb{H}^2=0$ may not be satisfied. However, we can define a possible singular analytic space 
\begin{center}
$\mathcal{S}:=\{t\in\Delta : H[\epsilon(t),\epsilon(t)]_{T(E)}=0\}.$
\end{center}
Let $X_{\epsilon(t)}, E_{\epsilon(t)}, \theta_{\epsilon(t)}$ be the complex manifold, the holomorphic bundle, and the Higgs field which ${\epsilon(t)}$ defines. By the above results, we have a family of holomorphic-Higgs pair $\{(X_{\epsilon(t)}, E_{\epsilon(t)},\theta_{\epsilon(t)})\}_{t\in\mathcal{S}}$. We call this family the \textit {Kuranishi family} of $(X, E,\theta)$ and $\mathcal{S}$ the \textit{Kuranishi space}. 

\subsection{Local completeness of Kuranishi family}
This section aims to give proof of the local completeness of the Kuranishi family for the deformation of the pair $(X, E,\theta)$. Here we follow Kuranishi's method.\par
Recall that in section \ref{Kuranishi family} we proved that for a given $\epsilon_1(t)=\sum_i t_i\eta_i\in\mathbb{H}^1=\mathrm{ker}(\Delta_{T(E)}:L^1\to L^1)$ the existence of solutions $\epsilon(t)$ to 
\begin{center}
$\epsilon(t)=\epsilon_{1}(t)+\dfrac{1}{2}d_{T(E)}^{*}G[\epsilon(t),\epsilon(t)]_{T(E)}$
\end{center}
and proved that $\epsilon(t)$ satisfies the Maurer-Cartan equation if and only if $H[\epsilon(t),\epsilon(t)]_{T(E)}=0$. Hence we obtain a family of holomorphic-Higgs pairs over 
\begin{center}
$\mathcal{S}:=\{t\in\Delta : H[\epsilon(t),\epsilon(t)]_{T(E)}=0\}$.
\end{center}
Before we state the main theorem of this paper, we introduce the Sobolev norm for $L$ and collect some estimates.\par
First, let us recall the Sobolev norm on Euclid space. Let $U$ be an open subset of $\mathbb{R}^n$ and $f$ and $g$ be a complex-valued smooth function on $\overline U$. Here, $\overline U$ is a closure of $U$. We set,
\begin{equation*}
(f,g)_k := \sum_{|\alpha|<k}\int_{U}^{}D^{\alpha}f\cdot\overline {D^{\alpha}g} dx,
\end{equation*}
where we use the multi-index notation $\alpha=(\alpha_1,\dots,\alpha_n)$, $\alpha_i>0$, $|\alpha|=\Sigma_{i}\alpha_i$ and $D^{\alpha}=\bigg(\dfrac{\partial}{\partial x_1}\bigg)^{\alpha_1}\cdots\bigg(\dfrac{\partial}{\partial x_n}\bigg)^{\alpha_n}$.\par
Then we define $k$-th Sobolev norm $|\cdot|_k$ as,
\begin{equation}
|f|_k=|f|_k^{U}:=\sqrt{(f,f)_k}.
\end{equation}
Let $V$ be a relatively compact open subset of $U$. By \cite[Chapter 4, Lemma 3.1]{Morrow-Kodaira_book}, we have an estimate such that 
\begin{equation}
|fg|^{V}_k\leq c|f|_k^{U} \cdot |g|_k^{U}, (k\geq n+2),
\end{equation}
where $c$ is a constant.\par
By using a partition of unity and the metric of $E$ and $X$, we can define $k$-th Sobolev $|\eta|_k$ for any $\eta\in L^i=\oplus_{p+q=i}A^{p,q}(\mathrm{EndE}(E))\oplus A^{0,i}(TX).$ We list some estimates which we need.  Let $c_k$ be a constant, then the following estimates hold (See \cite{Morrow-Kodaira_book} for more details),
\begin{equation}
\begin{split}\label{estimate 4}
|[\phi,\psi]_k|\leq c_k|\phi|_{k+1}|\psi|_{k+1}&(k\geq 2n+2, dim_{\mathbb{C}}X=n),\\
|H\phi|&\leq c_k|\phi|_k,\\
|d_{T(E)}^* G\phi|_k&\leq c_k|\phi|_{k-1}.
\end{split}
\end{equation}\par
From now on, we choose a $k$ enough large such that the above estimates hold.\par
Let $\eta:=(A+B,\phi)\in L^1$ be a Maurer-Cartan element and assume $|\eta|_k$ is small enough so that $\eta$ can define a holomorphic-Higgs pair. Let $X_\eta, E_\eta, \theta_\eta$ be the complex manifold, the holomorphic bundle, and the Higgs field which $\eta$ defines. We denote this holomorphic-Higgs pair $(X_\eta, E_\eta,\theta_\eta)$.  Let $\eta'\in L^1$ be another Maurer-Cartan element and assume that $\eta'$ also defines a holomorphic-Higgs pair $(X_{\eta'}, E_{\eta'},\theta_{\eta'})$. We denote as $(X_\eta, E_\eta, \theta_\eta)\cong(X_{\eta'},E_{\eta'},\theta_{\eta'})$  when there is a biholomorphic map $F:X_\eta\to X_\eta'$, a holomorphic bundle isomorphism $\Phi:E_\eta\to E_\eta'$ which is compatible with $F$ and $\theta_\eta=\widehat{\Phi}^{-1}\circ F^*(\theta_\eta')\circ \widehat{\Phi}$ holds. Here $\widehat{\Phi}:E_\eta\to F^{*}(E_{\eta'})$ is the holomorphic bundle isomorphism induced by $\Phi$. $F^{*}(E_{\eta'})$ is the pull back of the bundle $E_{\eta'}$ by $F$. \par
Now we state the main theorem of this paper.
\begin{theorem}\label{main}
Let $\eta:=(A+B,\phi)\in L^1$ be a Maurer-Cartan element. If $|\eta|_k$ is small enough, then there exists some $t\in \mathcal{S}$ such that  $(X_\eta,E_\eta,\theta_\eta)\cong(X_{\epsilon(t)},E_{\epsilon(t)},\theta_{\epsilon(t)})$.
\end{theorem}

\begin{proposition}\label{small sol is unique}
Let $\epsilon_1(t)\in\mathbb{H}^1$, $t\in S$. Assume that $\epsilon$  solves the equation, 
\begin{center}
$\epsilon=\epsilon_{1}(t)+\dfrac{1}{2}d_{T(E)}^{*}G[\epsilon,\epsilon]_{T(E)}.$
\end{center}
If $|\epsilon|_k$ is small enough, then the solution is unique.
\end{proposition}
\begin{proof}
Suppose $\epsilon$ is another solution. 
Let $\delta=\epsilon-\epsilon(t).$ Then
\begin{equation*}
\begin{split}
\delta&=\dfrac{1}{2}d^{*}_{T(E)}G([\epsilon,\epsilon]_{T(E)}-[\epsilon(t),\epsilon(t)]_{T(E)})\\
&=\dfrac{1}{2}d^{*}_{T(E)}G([\delta,\epsilon(t)]_{T(E)}+[\epsilon(t),\delta]_{T(E)}+[\delta,\delta]_{T(E)})\\
&=\dfrac{1}{2}d^{*}_{T(E)}G(2[\delta,\epsilon(t)]_{T(E)}+[\delta,\delta]_{T(E)})
\end{split}
\end{equation*}
Estimating $|\delta|_k$ gives
\begin{equation*}
\begin{split}
|\delta|_k&\leq D_k(|\delta|_k|\epsilon(t)|_k+|\delta|^{2}_k)\\
&\leq D_k|\delta|_k(|\epsilon(t)|_k+|\delta|_k).
\end{split}
\end{equation*}
If $|\epsilon(t)|_k$ is small enough, the above estimate hold if and only if $|\delta|_k=0.$
This proves the proposition.
\end{proof}

\begin{proposition}
Suppose $\eta\in L^1$ satisfies the Maurer-Cartan equation (\ref{eq DGLA for Higgs bundle}).
If $d^{*}_{T(E)}\eta=0$ and $|\eta|_k$ is small enough, then $\eta=\epsilon(t)$ for some $t\in S$.
\end{proposition}
\begin{proof}
Since $\eta$ satisfies the Maurer-Cartan equation, we have 
\begin{center}
$d_{T(E)}\eta-\dfrac{1}{2}[\eta,\eta]_{T(E)}=0.$
\end{center}
Since $d^{*}_{T(E)}\eta=0$, we have
\begin{equation*}
\begin{split}
\Delta_{T(E)}\eta&=d^{*}_{T(E)}d_{T(E)}\eta + d_{T(E)}d^{*}_{T(E)} \eta\\
&=\dfrac{1}{2}d^{*}_{T(E)}[\eta,\eta]_{T(E)}.
\end{split}
\end{equation*}
Hence
\begin{equation*}
\eta-H\eta=G\Delta_{T(E)}\eta=\dfrac{1}{2}Gd^{*}_{T(E)}[\eta,\eta]_{T(E)}.
\end{equation*}
Let $\psi:=H\eta$. Then $\eta=\psi +\dfrac{1}{2}Gd^{*}_{T(E)}[\eta,\eta]_{T(E)}$. By the assumption such that $|\eta|_{k}$ is small, 
 $|\psi|_k$ is small by \eqref{estimate 4}. Hence $\psi=\epsilon_1(t)$ for $|t|$ small enough. Hence by the Proposition $\ref{small sol is unique},$ $\eta=\epsilon(t)$ for some $t\in S.$
\end{proof}
In general $d_{T(E)}^*\eta\neq0$ so we must try something else. We follow the idea of \cite{Ku}. Let us recall how we solved this problem in the complex manifold setting. 
The idea is that for a given Maurer-Cartan element $\phi\in A^{0,1}(TX)$, we deform $\phi$ along a diffeomorphism $f:X\to X$. \par
Let  $X_\phi$  be a complex manifold such that the complex structure comes from $\phi$. 
Let $f: X\to X$ be a diffeomorphism.   
We can induce a complex structure on $X$ by $f$.  We denote the corresponding Maurer-Cartan element as $\phi\circ f$. Note that $f:X_{\phi\circ f}\to X_{\phi}$ is a biholmorphic map.\par
Kuranishi showed that for every Maurer-Cartan elements $\phi$ with $|\phi|_k$  small, there is a diffeomorphism $f$ such that $\overline\partial_{TX}(\phi\circ f)=0.$ We recall  how we obtain such $f$.\par
Let $g=(g_{i\overline j})$ be a fixed Hermitian metric on $X$. Then we have the Christoffel symbols,
 \begin{equation*}
 \Gamma^i_{k,j}(z)=\sum_{l}g^{\overline l i}(z)\bigg(\dfrac{\partial g_{j\overline l}}{\partial z_k}\bigg)(z).
 \end{equation*}
Let $\xi=\sum_i\xi_i(z)\dfrac{\partial}{\partial z_i}\in A(TX).$ Let $z_0\in X$. Let $c(t)=c(t,z_0,\xi)=(c_1(t),\dots,c_n(t))$ be a curve such that,
\begin{itemize}
      \item $c(0,z_0,\xi)=z_0$,
      \item $\dfrac{dc}{dt}(0)=\sum_i\xi_i(z_0)\dfrac{\partial}{\partial z_i}$,
      \item $\dfrac{d^2 c_i}{dt^2}(t)+\sum_{j,k}\Gamma^i_{j,k}(c(t))\dfrac{dc_j}{dt}\dfrac{dc_k}{dt}=0$.
\end{itemize}
Let $f_\xi(z_0):=c(1,z_0,\xi)$. Since $X$ is compact, $f_\xi$ is a diffeomorphism. By using Taylor expansion for $f_\xi$, we obtain,
\begin{equation}\label{Kur eq}
\phi\circ f_\xi =  \phi + \overline\partial_{TX}\xi +R(\phi,\xi)
\end{equation}
where $R(t\phi,t\xi)=t^2R_1(\phi,\xi,t)$ if $t$ is a real number and both $R,R_1$ are smooth map on $X$. In \cite{Ku}, it was shown that there is a $\xi\in A(TX)$ such that $\overline\partial_{TX}(\phi\circ f_\xi)=0$ for any $\phi$ with $|\phi|_k$ small by the implicit function theorem between Banach spaces. \par
Let $\eta=(A+B,\phi)\in L^1$ be a Maurer-Cartan elements and assume $|\eta|_k$ is small enough so that $\eta$ can define a holomorphic-Higgs pair $(X_\eta,E_\eta,\theta_\eta)$. By the Kuranishi's work we have a $\xi\in A(TX)$ such that $\overline\partial_{TX}(\phi\circ f_\xi)=0$.  
\par
Let $P_\xi: E\to E$ be the parallel transport of the Chern connection along $f_\xi$. Let $\upsilon\in A(\mathrm{End}(E))$ and $exp(\upsilon):=\sum_{n=0}^{\infty}\dfrac{\upsilon^n}{n!}\in A(\mathrm{End}(E))$. $exp(\upsilon):E\to E$ is  an automorphism and the inverse is given as $exp(-\upsilon)$. Note that $(\upsilon,\xi)\in L^0.$ \par
Let $\Phi:=P_\xi\circ exp(\upsilon)$. Since $P_\xi$ is an isomorphism and compatible with $f_\xi$, $\Phi$ also is. Hence there is a smooth bundle isomorphism  $\widehat{\Phi}: E\to f_\xi^* E_\eta$ which is induced by $\Phi$. Hence we can induce a holomorphic-Higgs pair structure on $(X, E,\theta)$ via $\Phi$ and $f_\xi$. This holomorphic-Higgs pair is isomorphic to $(X_\eta, E_\eta,\theta_\eta)$. We denote the corresponding Maurer-Cartan element as $\eta_\gamma:=((A+B)\star\Phi,\phi\circ f_\xi)$.   
We show the existence of $\gamma:=(\upsilon,\xi)\in L^0$ such that $d_{T(E)}^*\eta_\gamma=0$. \par
We first prove the next proposition.
\begin{proposition}\label{small deform}
Let  $\eta_\gamma=((A+B)\star\Phi,\phi\circ f_\xi)$, $\eta=(A+B,\phi)$ and $\gamma=(\upsilon,\xi)$ be as above. Then we have,
\begin{equation}
((A+B)\star\Phi,\phi\circ f_\xi)=(A+B,\phi)+d_{T(E)}(\upsilon,\xi) + R((A,B,\phi),(\upsilon,\xi)).
\end{equation}
The error term $R$ is of order $t^2$ in the sense that 
\begin{equation*}
R(t(A,B,\phi),t(\upsilon,\xi))=t^2R_1((A,B,\phi),(\upsilon,\xi),t).
\end{equation*}
where $t$ is a real number and $R_1$ is a smooth map. 
\end{proposition}
\begin{proof}
Before going to the proof, we prepare some terminologies. Let $A\in A^{0,1}(\mathrm{End}(E)), B\in A^{1,0}(\mathrm{End}(E)), \upsilon\in A(\mathrm{End}(E)), \phi\in A^{0,1}(TX)$ and $\xi\in A(TX)$. If we denote $R((A, B,\phi),(\upsilon,\xi))$ then it is a smooth map on $X$ such that $R$ depends on $A, B,\upsilon,\phi$ and $\xi$ and $R$ is of order $t^2$ in the sense that
\begin{equation*}
R(t(A,B,\phi),t(\upsilon,\xi))=t^2R_1((A,B,\phi),(\upsilon,\xi),t)
\end{equation*}
where $t$ is a real number and $R_1((A,B,\phi),(\upsilon,\xi),t)$ is a smooth map defined on $X$. We assume that the same property holds for $R((A,\phi),(\upsilon,\xi)), R((B,\phi),(\upsilon,\xi))$.\par
If we denote $R'((A, B,\phi),(\upsilon,\xi))$, then it is a smooth map defined on some open set of $X$ such that  $R'$ depends on $A, B,\upsilon,\phi$, and $\xi$ and $R'$ is of order $t^2$ in the sense that 
\begin{equation*}
R'(t(A,B,\phi),t(\upsilon,\xi))=t^2R'_1((A,B,\phi),(\upsilon,\xi),t)
\end{equation*}
where $t$ is a real number and $R'_1((A,B,\phi),(\upsilon,\xi),t)$ is a smooth map defined on some open set of $X$. We assume that the same property holds for $R'((A,\phi),(\upsilon,\xi)), R'((B,\phi),(\upsilon,\xi)),R'(\upsilon,\xi)$ and $R'(\phi,\xi)$.\par
By \eqref{Kur eq}, we only have to prove the following.
\begin{equation}\label{eq00}
(A+B)\star\Phi=A + B + \overline\partial_{\mathrm{End}(E)}\upsilon + \xi\lrcorner F_{d_K} +[\theta,\upsilon] + \{\partial_{K}^{\mathrm{End}(E)},\xi\lrcorner\}\theta + R((A,B,\phi),(\upsilon,\xi)).
\end{equation}
First we prove
\begin{equation}\label{eq0}
A\star\Phi=A + \overline\partial_{\mathrm{End}(E)}\upsilon + \xi\lrcorner F_{d_K} + R((A,\phi),(\upsilon,\xi)).
\end{equation}\par
Let $U'$ and $U$ be open sets of $X$ such that $U'\subset U$ and $f_\xi(U')\subset U$. We calculate $(A\star\Phi-A)(z)$ for $z\in U'$.
Let $\{e_k\}$ be a holomorphic frame on $U$ for $E_{\eta'}$.  Since $E_{\eta'}$'s complex structure is induced by $\Phi$, $\Phi:E_{\eta'}\to E_\eta$ is a holomorphic bundle isomorphism. Hence $\Phi(e_k)$ is a holomorphic section for $E_\eta$. Hence we have,
\begin{align}
\label{eq1}\overline\partial_{E}e_k +\{\partial_K,(\phi\circ f_\xi)\lrcorner\}e_k + (A\star\Phi) e_k=0,\\
\label{eq2}\Phi^{-1}\circ(\overline\partial_{E}+\{\partial_k,\phi\lrcorner\} + A)\circ\Phi(e_k)=0.
\end{align}
Since \eqref{Kur eq}, \eqref{eq1} is equivalent to,
\begin{equation}\label{eq3}
\overline\partial_{E}e_k +\{\partial_K,\phi\lrcorner\}e_k +\{\partial_K, \overline\partial_{TX}\xi\lrcorner\}e_k + (A\star\Phi) e_k +  R'(\phi,\xi)(e_k)=0.
\end{equation}
Let $P'_\xi$ be the first order of $P_\xi$.
Since $\Phi=P_\xi\circ exp(\upsilon)$, we have an expansion for $\Phi (e_k)$ such that
\begin{equation*}
\Phi(e_k)=P_\xi\circ exp(\upsilon)(e_k)=e_k+P'_\xi(e_k)+\upsilon(e_k)+R'(\upsilon,\xi)(e_k).
\end{equation*}
Hence \eqref{eq2} is equivalent to,
\begin{equation*}
\overline\partial_{E}e_k + \{\partial_K,\phi\lrcorner\}e_k + Ae_k+\overline\partial_{E}(P'_\xi e_k) - P'_\xi\overline\partial_E e_k + \overline\partial_{E}\upsilon e_k -\upsilon\overline\partial_E(e_k) +R'((A,\phi),(\upsilon,\xi))(e_k)=0.
\end{equation*}
Since $\overline\partial_E(\upsilon e_k)-\upsilon\overline\partial_E e_k=(\overline\partial_{\mathrm{End}(E)}\upsilon)e_k,$ we have,
\begin{equation}\label{eq4}
\overline\partial_{E}e_k + \{\partial_K,\phi\lrcorner\}e_k + Ae_k + \overline\partial_{E}(P'_\xi e_k) - P'_\xi\overline\partial_E e_k  + (\overline\partial_{\mathrm{End}(E)}\upsilon )e_k +R'((A, \phi),(\upsilon,\xi))(e_k)=0.
\end{equation}
Hence by \eqref{eq3}, \eqref{eq4}, we have,
\begin{equation}\label{eq5}
(A\star\Phi)e_k - Ae_k +\{\partial_K, \overline\partial_{TX}\xi\lrcorner\}e_k - \overline\partial_{E}(P'_\xi e_k) + P'_\xi\overline\partial_E e_k- (\overline\partial_{\mathrm{End}(E)}\upsilon)e_k +R'((A,\phi),(\upsilon,\xi))(e_k)=0.
\end{equation}
We have to prove $\{\partial_K, \overline\partial_{TX}\xi\lrcorner\} - \overline\partial_{E}\circ P'_\xi + P'_\xi\circ\overline\partial_E =-\xi\lrcorner F_{d_K}$. We prove this for a holomorphic frame $\{e'_k\}$ for $E$ on $U$. 
Since $P_\xi$ is the parallel transport along $f_\xi$ respect to the Chern connection,  we have 
$P'_\xi(e'_k)=-\xi\lrcorner K^{-1}\partial K(e'_k)$. (See \cite{Spi} for more details).
Hence we have, 
\begin{equation*}
\begin{split}
\{\partial_K, \overline\partial_{TX}\xi\lrcorner\}e'_k - \overline\partial_{E}\circ P'_\xi e'_k + P'_\xi\circ\overline\partial_E e'_k&=-\overline\partial_{TX}\xi\lrcorner\partial_K e'_k +\overline\partial_E(\xi\lrcorner K^{-1}\partial Ke'_k )\\ 
&=-\overline\partial_{TX}\xi\lrcorner\partial_K e'_k +\overline\partial_{TX}\xi\lrcorner K^{-1}\partial Ke'_k -\xi\lrcorner(\overline\partial_{\mathrm{End}(E)}(K^{-1}\partial K))e'_k\\
&=(-\xi\lrcorner F_{d_K})e'_k.
\end{split}
\end{equation*}
Hence by \eqref{eq5}, we have 
\begin{equation*}\label{eq6}
(A\star\Phi)e_k = Ae_k  + (\overline\partial_{\mathrm{End}(E)}\upsilon)e_k + \xi\lrcorner F_{d_K}(e_k) +R'((A,\phi),(\upsilon,\xi))(e_k).
\end{equation*}
Since $\{e_k\}$ is an arbitrary holomorphic frame on $E_{\eta'}$ and $R'((A,\phi),(\upsilon,\xi))(e_k)$ is a local expression of $A\star\Phi-A - \overline\partial_{\mathrm{End}(E)}\upsilon -\xi\lrcorner F_{d_K}$ we proved \eqref{eq0}.\par
Next, we prove that,
\begin{equation}\label{eq6}
B\star\Phi= B + [\theta,\upsilon] + \{\partial_{K}^{\mathrm{End}(E)},\xi\lrcorner\}\theta +R((B,\phi),(\upsilon,\xi)).
\end{equation}
Recall that $\widehat{\Phi}:E_{\eta'}\to f_\xi^*(E_\eta)$ is a holomorphic bundle isomorphism and $\theta_{\eta'}=\widehat{\Phi}^{-1}\circ f_\xi^*(\theta_\eta)\circ\widehat{\Phi}$. Let $\theta_{\eta'}^{1,0}$ is the $(1,0)$-part of $\theta_{\eta'}$ respect to the original complex structure, then we have $B\star\Phi=\theta_{\eta'}^{1,0}-\theta$.  We calculate $B\star\Phi$ locally.\par
Let $(U,z)$ be a local coordinate and $U'\subset U$. We assume $f_\xi(U')\subset U$ and $\xi=\sum_i\xi_i(z)\bigg(\dfrac{\partial}{\partial z_i}\bigg)$. By the definition of $f_\xi$, for $z\in U'$, we have
 \begin{equation*}
f_\xi(z)=(z_1+\xi_1(z)+O(|\xi|^2),\dots, z_n+\xi_n(z)+O(|\xi|^2)).
\end{equation*}
Let $\{e_k\}$ be a holomorphic frame on $U$ for $E$. Let $g_i, B_i\in A(\mathrm{End}(E)).$ Assume that $\theta$ is locally expressed as $\sum_i g_i(z)dz_i$ and $B$ as $\sum_i B_i(z)dz_i$ respect to this frame. \par
Let the bracket $[\cdot,\cdot]$ be the canonical Lie bracket defined on $A^*(\mathrm{End}(E)).$ 
Since $\theta_\eta=\theta+ B +\phi\lrcorner(\theta + B)=(g_i+B_i) dz_i+(g_i+B_i)\phi^i_j d\overline z_j$ and $\widehat{\Phi}$ is induced by $\Phi$, we have,
\begin{equation*}
\begin{split}
\theta_{\eta'}(e_k)=&\widehat{\Phi}^{-1}\circ f_\xi^*(\theta_\eta)\circ\widehat{\Phi}(e_k)\\
=&\widehat{\Phi}^{-1}\circ\{(g_i+B_i)(f_\xi(z))d f_{\xi,i}(z)+(g_i+B_i)(f_\xi(z))\phi^i_j(f_\xi(z))d\overline f_{\xi,j}\}\circ\widehat{\Phi}(e_k)\\
=&((g_i+B_i)(f_\xi(z))d f_{\xi,i}(z)+(g_i+B_i)(f_\xi(z))\phi^i_j(f_\xi(z))d\overline f_{\xi,j})(e_k)\\
&+[(g_i+B_i)(f_\xi(z))d f_{\xi,i}(z)+(g_i+B_i)(f_\xi(z))\phi^i_j(f_\xi(z))d\overline f_{\xi,j},P'_\xi](e_k)\\
&+[(g_i+B_i)(f_\xi(z))d f_{\xi,i}(z)+(g_i+B_i)(f_\xi(z))\phi^i_j(f_\xi(z))d\overline f_{\xi,j},\upsilon](e_k)+R'((B,\phi),(\upsilon,\xi))(e_k).
\end{split}
\end{equation*}
Hence,
\begin{equation}\label{eq7}
\begin{split}
 \theta_{\eta'}^{1,0}(e_k)=&\bigg((g_i+B_i)(f_\xi(z))dz_i+(g_i+B_i)(f_\xi(z))\dfrac{\partial \xi_i}{\partial z_j}dz_j\bigg)(e_k)\\
 &+\bigg[((g_i+B_i)(f_\xi(z))dz_i+(g_i+B_i)(f_\xi(z))\dfrac{\partial \xi_i}{\partial z_j}dz_j,P'_\xi\bigg](e_k)\\
 &+\bigg[((g_i+B_i)(f_\xi(z))dz_i+(g_i+B_i)(f_\xi(z))\dfrac{\partial \xi_i}{\partial z_j}dz_j,\upsilon\bigg](e_k) +R'((B,\phi),(\upsilon,\xi))(e_k)\\
 =&\bigg((g_i+B_i)(f_\xi(z))dz_i+g_i(f_\xi(z))\dfrac{\partial \xi_i}{\partial z_j}dz_j\bigg)(e_k)\\
 &+[g_i(f_\xi(z))dz_i,P'_\xi](e_k)+[g_i(f_\xi(z))dz_i,\upsilon](e_k)+R'((B,\phi),(\upsilon,\xi))(e_k).
 \end{split}
 \end{equation}
Since $f_\xi(z)=c(z,\xi,1),$ we have the Taylor expansion at $t=0$ for $g_i(f_\xi(z))$ and $B_i(f_\xi(z))$.
\begin{equation*}
\begin{split}
g_i(f_\xi(z))=g_i(z)+\xi_j(z)\dfrac{\partial g_i}{\partial z_j}(z)+O(|\xi|^2),\\
B_i(f_\xi(z))=B_i(z)+\xi_j(z)\dfrac{\partial B_i}{\partial z_j}(z)+O(|\xi|^2).
\end{split}
\end{equation*} 
Hence by \eqref{eq7},  $\theta_{\eta'}^{1,0}(e_k)$ becomes,
\begin{equation*}
\begin{split}
&\theta_{\eta'}^{1,0}(e_k)\\
=&\bigg(g_i(z)dz_i + \xi_j\dfrac{\partial g_i}{\partial z_j}(z)dz_i+B_i(z)dz_i+g_i(z)\dfrac{\partial \xi_i}{\partial z_j}dz_j\bigg)(e_k)
+[g_i(z)dz_i,P'_\xi](e_k)+[g_i(z)dz_i,\upsilon](e_k)+R'((B,\phi),(\upsilon,\xi))(e_k)\\
=&(\theta+B)(e_k)+\bigg( \xi_j\dfrac{\partial g_i}{\partial z_j}(z)dz_i+g_i(z)\dfrac{\partial \xi_i}{\partial z_j}dz_j\bigg)(e_k)+[\theta,P'_\xi](e_k)+[\theta,\upsilon](e_k)+R'((B,\phi),(\upsilon,\xi))(e_k).
\end{split}
\end{equation*}
Hence,
\begin{equation}\label{eq8}
\begin{split}
B\star\Phi(e_k)=&\theta^{1,0}_\eta(e_k)-\theta(e_k)\\
=&B(e_k)+\bigg( \xi_j\dfrac{\partial g_i}{\partial z_j}(z)dz_i+g_i(z)\dfrac{\partial \xi_i}{\partial z_j}dz_j\bigg)(e_k)+[\theta,P'_\xi](e_k)+[\theta,\upsilon](e_k)+R'((B,\phi),(\upsilon,\xi))(e_k).
\end{split}
\end{equation}
Hence the only thing we have to prove is  $ \bigg(\xi_j\dfrac{\partial g_i}{\partial z_j}(z)dz_i+g_i(z)\dfrac{\partial \xi_i}{\partial z_j}dz_j\bigg)(e_k)+[\theta,P'_\xi](e_k)=(\{\partial_{K}^{\mathrm{End}(E)},\xi\lrcorner\}\theta)(e_k).$ Since $\{e_k\}$ is a local holomorphic frame for $E$, we have 
\begin{equation*}
\begin{split}
&\bigg(\xi_j\dfrac{\partial g_i}{\partial z_j}(z)dz_i+g_i(z)\dfrac{\partial \xi_i}{\partial z_j}dz_j\bigg)(e_k)+[\theta,P'_\xi](e_k)\\
=&\bigg(\xi_j\dfrac{\partial g_i}{\partial z_j}(z)dz_i+g_i(z)\dfrac{\partial \xi_i}{\partial z_j}dz_j\bigg)(e_k)+\theta(-\xi\lrcorner K^{-1}\partial K)(e_k)+\xi\lrcorner K^{-1}\partial K(\theta(e_k))\\
=&\bigg(\xi_j\dfrac{\partial g_i}{\partial z_j}(z)dz_i+g_i(z)\dfrac{\partial \xi_i}{\partial z_j}dz_j\bigg)(e_k)+[\xi\lrcorner K^{-1}\partial K,\theta](e_k)\\
\mathrm{and}&\\
&\{\partial_K,\xi\lrcorner\}\theta=\partial_K(\xi\lrcorner\theta)+\xi\lrcorner\partial_K(\theta)=\partial(\xi_i(z)g_i)+[K^{-1}\partial K,\xi\lrcorner\theta]+\xi_j(z)\dfrac{\partial g_i}{\partial z_j}dz_i-\xi_i(z)\dfrac{\partial g_i}{\partial z_j}dz_j+\xi\lrcorner[K^{-1}\partial K,\theta]\\
=&\dfrac{\partial}{\partial z_j}(\xi_i(z)g_i(z))dz_j+\xi_j(z)\dfrac{\partial g_i}{\partial z_j}dz_i-\xi_i(z)\dfrac{\partial g_i}{\partial z_j}dz_j+[\xi\lrcorner K^{-1}\partial K,\theta]\\
=&\bigg(\xi_j\dfrac{\partial g_i}{\partial z_j}(z)dz_i+g_i(z)\dfrac{\partial \xi_i}{\partial z_j}dz_j\bigg)+[\xi\lrcorner K^{-1}\partial K,\theta].
\end{split}
\end{equation*}
Hence we have the desired equality. Hence by \eqref{eq8} we have,
\begin{equation*}
B\star\Phi(e_k)=B(e_k)+[\theta,\upsilon](e_k)+(\{\partial_{K}^{\mathrm{End}(E)},\xi\lrcorner\}\theta)(e_k)+R'((B,\phi),(\upsilon,\xi))(e_k).
\end{equation*}
Since $\{e_k\}$ is an arbitrary local holomorphic frame on $E$ and $R'((B,\phi),(\upsilon,\xi))(e_k)$ is a local expression of $B\star\Phi- B - [\theta,\upsilon]-(\{\partial_{K}^{\mathrm{End}(E)},\xi\lrcorner\}\theta)$, we proved \eqref{eq6}.\par
Hence by \eqref{eq0} and \eqref{eq6} we proved \eqref{eq00}. This completes the proof.
\end{proof}
Recall that $\mathbb{H}^0=\mathrm{ker}(\Delta_{T(E)})\subset L^0$. Let $F^0$ be the orthogonal complement of $\mathbb{H}^0$ w.r.t the inner product $(\cdot,\cdot)$. Note that $\mathrm{ker}(H)=F^0$. $H$ is the harmonic projection. Then for $\gamma\in F^0$,
\begin{equation*}
\eta=G\Delta_{T(E)}\gamma+H\gamma=G\Delta_{T(E)}\gamma.
\end{equation*}
Since $d^{*}_{T(E)}$ is zero on $L^0$, $d^{*}_{T(E)}(\gamma)=0$. Hence
\begin{equation*}
\Delta_{T(E)}\gamma=d^{*}_{T(E)}d_{T(E)}\gamma.
\end{equation*}
This yields,
\begin{equation}\label{orthogonal}
\gamma=Gd^{*}_{T(E)}d_{T(E)}\gamma.
\end{equation}
From now on, we think $L^1, F^0$ as normed by $k$-th Sobolev norm and $L^0$ normed by $(k-1)$-th Sobolev norm. Let $L^0_{k-1}, L^1_k, F^0_k$ be the completion of $L^0,L^1,F^0$ with respect to the corresponding norms.
\begin{proposition}\label{implicit fun}
Let $\eta(\gamma):=\eta+d_{T(E)}\gamma+R(\eta,\gamma).$ There are neighborhoods of the origin $U$ and $V$ in $L^0$ and $F^0$ such that for any $\eta\in U$ there is a $\gamma\in V$ such that
\begin{equation}\label{vanish}
d^*_{T(E)}(\eta(\gamma))=0
\end{equation}
\end{proposition}
\begin{proof}
Let $\gamma:=(\upsilon,\xi)\in F^0$. By the definition of $\eta(\gamma)$, \eqref{vanish} is equivalent to 
\begin{equation*}
0=d^*_{T(E)}(\eta(\gamma))=d^*_{T(E)}\eta+d^*_{T(E)}d_{T(E)}\gamma+d^*_{T(E)}R(\eta,\gamma).
\end{equation*}
By \eqref{orthogonal},
\begin{equation*}
\gamma=Gd^{*}_{T(E)}d_{T(E)}\gamma=-Gd^*_{T(E)}\eta-Gd^*_{T(E)}R(\eta,\gamma).
\end{equation*}
Thus \eqref{vanish} is equivalent to
\begin{equation*}
\gamma+Gd^*_{T(E)}\eta+Gd^*_{T(E)}R(\eta,\gamma)=0.
\end{equation*}
Let $U_1$ and $V_1$ are neighborhoods of the origin of $L^1_k$ and $F^0_k$. By the local property of $R(\eta,\gamma)$ which we observed in Proposition \ref{small deform}, we can define a map $C^1$ map $ h:U_1\times V_1\to L^0_{k-1}$ by,
\begin{equation*}
h(\eta,\gamma)=\gamma+Gd^*_{T(E)}\eta+Gd^*_{T(E)}R(\eta,\gamma).
\end{equation*}
By the order condition on the error term $R$, the identity map is the derivative of $h$ concerning $\gamma$ at the $(0,0)$. Hence by the implicit function theorem for Banach spaces, there exists an open neighborhood $U_0$ of $0\in L_k^1$ and a  continuous map $g: U_0\to V_1$ such that $g(0)=0$ and such that $h(\eta,\gamma)=0$  if and only if $\gamma=g(\eta)$ for all $\eta\in U_0$ (See \cite{Lang} for details).\par
Let $U:=U_0\cap L^0$ and $V:=g(U_0)\cap F^0$. Let $\eta\in U$ and $\gamma:=g(\eta)$. By the previous section, we have $h(\eta,\gamma)=0.$ If we take $U_0$ small enough, $\Delta_{T(E)}+d^*_{T(E)}R(\eta,\cdot)+d^*_{T(E)}\eta$ is a quasi-linear elliptic operator. By elliptic regularity, $\gamma$ is smooth. Hence $\gamma\in V$. Hence this completes the proof.
\end{proof}
We can now give the proof of Theorem \ref{main}.
\begin{proof}[Proof of Theorem \ref{main}] Let $\eta\in L^1$ be a Maurer-Cartan element and $|\eta|_k<<1$. By Proposition 4.4, we only have to prove for $d^*_{T(E)}\eta\neq 0$. By Proposition \ref{implicit fun}, we have a $\gamma=(\upsilon,\xi)\in L^0$ such that 
\begin{equation*}
d^*_{T(E)}\eta+d^*_{T(E)}d_{T(E)}\gamma+d^*_{T(E)}R(\eta,\gamma)=0.
\end{equation*}
Let $\Phi:=P_\xi\circ exp(\upsilon)$. We can induce a holomorphic-Higgs structure on $(X,E,\theta)$ that is isomorphic to $(X_\eta, E_\eta, \theta_\eta)$ by $\Phi$ and $f_\xi$. We denote the corresponding Maurer-Cartan element as $\eta_\gamma$. By Proposition \ref{small deform}, we have 
\begin{equation*}
\eta_\gamma=\eta+d_{T(E)}\gamma+R(\eta,\gamma).
\end{equation*}
We can easily see that $d^*_{T(E)}\eta_\gamma=0$. Hence by Proposition 4.4, we have some $t\in \mathcal{S}$ such that  $(X_{\epsilon(t)}, E_{\epsilon(t)},\theta_{\epsilon(t)})\cong (X_\eta, E_\eta, \theta_\eta).$ This completes the proof.
\end{proof}

\newcommand{\Addresses}
  \bigskip
  \footnotesize
  T.~Ono, \textsc{Department of Mathematics, Graduate School of Science, Osaka University, Osaka,
Japan.}\par\nopagebreak
  \textit{E-mail address}: \texttt{u708091f@ecs.osaka-u.ac.jp}

\end{document}